\documentclass[11pt]{amsart}
\usepackage{amssymb}
\usepackage{amsmath}
\usepackage{mathrsfs}
\usepackage{faktor}
\usepackage{tikz}

\usetikzlibrary{arrows}
\usetikzlibrary{shapes.geometric}
\usetikzlibrary{decorations.pathreplacing,decorations.markings}
\usetikzlibrary{automata,positioning,arrows,calc}
\usepackage[utf8]{inputenc}
\usepackage[top=1in,bottom=1in,right=1in,left=1in]{geometry}
\usepackage{hyperref}
\usepackage{enumitem}

\newtheorem{theorem}{Theorem}[section]
\newtheorem{lem}[theorem]{Lemma}
\newtheorem{proposition}[theorem]{Proposition}
\newtheorem{cor}[theorem]{Corollary}

\theoremstyle{definition}
\newtheorem{dfn}[theorem]{Definition}
\newtheorem{ex}[theorem]{Example}
\newtheorem{rmk}[theorem]{Remark}
\newtheorem{ntn}[theorem]{Notation}

\numberwithin{theorem}{section}

\newenvironment{theorem_no_number}[1][]{\begin{trivlist}
\item[\hskip \labelsep {\bfseries Theorem \def\temp{#1}\ifx\temp\empty  #1\else  #1\fi
.}] \itshape}  {\end{trivlist}}

\renewenvironment{proof}[1][]{\begin{trivlist}
\item[\hskip \labelsep {\bfseries Proof  \def\temp{#1}\ifx\temp\empty  #1\else  (#1)\fi
}]}{\hfill\(\square\) \end{trivlist}}

\DeclareMathOperator{\id}{id}

\DeclareMathOperator{\im}{im}

\DeclareMathOperator{\Aut}{Aut}
\DeclareMathOperator{\End}{End}

\DeclareMathOperator{\GL}{GL}

\DeclareMathOperator{\FIN}{FIN}

\newcommand{\ns}{\mathsf{NSPACE}}

\title{EDT0L solutions to equations in group extensions}

\author{Alex Levine}

\address{Department of Mathematics, Alan Turing Building, The University
of Manchester, Manchester M13 9PL, UK}

\email{alex.levine@manchester.ac.uk}



\begin{document}

\begin{abstract}
	We show that the class of groups where EDT0L languages can be used to describe
	solution sets to systems of equations is closed under direct products,
	wreath products with finite groups, and passing to finite index subgroups. We
	also add the class of groups that contain a direct product of hyperbolic
	groups as a finite index subgroup to the list of groups where solutions to
	systems of equations can be expressed as an EDT0L language. This includes
	dihedral Artin groups. We also show that the systems of equations with
	rational constraints in virtually abelian groups have EDT0L solutions, and the
	addition of recognisable contraints to any system preserves the property of
	having EDT0L solutions. These EDT0L solutions are expressed with respect to
	quasigeodesic normal forms. We discuss the space complexity in which EDT0L
	systems for these languages can be constructed.
\end{abstract}

\maketitle

\section{Introduction}


	Formal languages have been successfully employed for the last 40 years to
	describe important sets in groups, in order to restrict the use of memory in
	algorithms describing these sets, and give the sets a clear structure. The use
	of languages to represent solutions to equations made a leap when Ciobanu,
	Diekert and Elder proved that the sets of solutions to systems of equations in
	free groups with rational constraints can be expressed as EDT0L languages
	\cite{eqns_free_grps}. Solutions to systems of equations in right-angled Artin
	groups were then shown to be EDT0L by Diekert, Je\.{z} and Kufleitner
	\cite{EDT0L_RAAGs}. Virtually free groups \cite{VF_eqns}, hyperbolic groups
	\cite{eqns_hyp_grps}, and virtually abelian groups \cite{VAEP} followed later.

 	In the 1960s, Lindenmayer introduced a collection of classes of languages
 	called \textit{L-systems}. Originally used to study growth of organisms,
 	L-systems saw significant interest in the 1970s and early 1980s, and
 	Lindenmayer's original classes inspired the definitions of many other
 	L-systems, including Rozenberg's EDT0L languages
 	\cite{ET0L_EDT0L_def_article}. This class has recently had a variety of
 	applications in and around group theory (\cite{appl_L_systems_GT},
 	\cite{EDT0L_permuations}, \cite{bishop_elder_journal},
 	\cite{more_than_1700}). For a comprehensive introduction to L-systems,
 	including  EDT0L languages, we refer the reader to
 	\cite{math_theory_L_systems}.

 Theorem A collects the main results in this paper. The format used to express
 solutions as words is explained in the preliminaries (Section
 \ref{prelim_sec}).

	\begin{theorem_no_number}[A]
		Let \(G\) and \(H\) be groups where solution languages to systems of
		equations are EDT0L, with respect to normal forms \(\eta_G\) and \(\eta_H\),
		respectively, and these EDT0L systems are constructible in \(\ns(f)\), for
		some \(f\). Then in the following groups, solutions to systems of equations
		are EDT0L, and an EDT0L system can be constructed in non-deterministic
		\(f\)-space:
		\begin{enumerate}
			\item \(G \wr F\), for any finite group \(F\) (Proposition
			\ref{wreath_product_EDT0L_prop});
			\item \(G \times H\) (Proposition \ref{dir_prod_EDT0L_prop});
			\item Any finite index subgroup of \(G\) (Proposition
			\ref{fin_index_EDT0L_prop});
		\end{enumerate}
		In the following groups, solutions to systems of equations are
		EDT0L, and an EDT0L system can be constructed in \(\ns(n^4 \log n)\):
		\begin{enumerate}[resume]
			\item Any group that is virtually a direct product of hyperbolic
			groups (Corollary \ref{main_thm_cor});
			\item Dihedral Artin groups (Corollary \ref{Artin_groups_cor}).
		\end{enumerate}
		If \(\eta_G\) and \(\eta_H\) are both quasigeodesic or regular, then the
		same will be true for the normal forms used in (1), (2) and (3). It is
		possible to choose normal forms for the groups that are virtually direct
		products of hyperbolic groups (4), and dihedral Artin groups in (5) that are
		regular and quasigeodesic.
	\end{theorem_no_number}

	Whilst an understanding of the set of solutions to a system of equations in a
	direct product follows immediately from understanding the solutions to the
	projection onto each of the groups in the direct product, showing that the
	language can be expressed in the correct format requires more work, which we
	explore in Section \ref{distinguished_letter_sec}. This format is also
	required to prove Theorem A(1).

	The proof of Theorem A(4) is based on Ciobanu, Holt and Rees' proof of the
	fact the satisfiability of systems of equations in these groups is decidable
	\cite{equations_VDP}, in a work that also looks at recognisable constraints.
	We show that the addition of recognisable constraints to any system of
	equations preserves the property of having an EDT0L solution language, and use
	this to show that the class of groups where systems of equations have EDT0L
	solutions is closed under passing to finite index subgroups.

	Equations with rational constraints have also attracted a large amount of
	attention; the addition of constraints allows for a certain level of control
	on what each of the variables can be. The fact that systems of equations in
	free and virtually free groups have EDT0L solutions was also shown to be true
	if rational constraints are added (\cite{eqns_free_grps}, \cite{VF_eqns}). In
	hyperbolic groups, the addition of rational constraints was shown to preserve
	the fact that systems of equations have EDT0L solutions if the rational
	constraints were quasi-isometrically embedded \cite{eqns_hyp_grps}. We
	generalise the result on systems of equations in virtually abelian groups in
	\cite{VAEP} to include rational constraints. This result uses virtually
	abelian equation length, introduced in \cite{VAEP}, which assigns a much
	smaller length to equations in virtually abelian groups than the standard
	length, and thus results using this length are stronger than equivalent
	results with standard length. This length is defined using the fact that
	equations in virtually abelian groups can be described using a number of
	integers, and the fact that integers can be stored in logarithmic space.

	\begin{theorem_no_number}[B]
		Solutions to a system of equations with rational constraints in a virtually
		abelian group are EDT0L in non-deterministic quadratic space, with respect
		to virtually abelian equation length, and with respect a regular and
		quasigeodesic normal form.
	\end{theorem_no_number}

	Group equations have seen significant interest as to the decidability of the
	satisfiability of equations in specific classes of groups, since Makanin
	showed in the 1980s that the satisfiability of systems in free groups was
	decidable (\cite{Makanin_systems}, \cite{Makanin_semigroups},
	\cite{Makanin_eqns_free_group}). The satisfiability of systems of equations
	has been shown to be decidable or undecidable in a wide variety of other
	classes of groups (\cite{romankov_univ_theory}, \cite{Sela},
	\cite{dahmani_guirardel}, \cite{duchin_liang_shapiro}, \cite{equations_VDP},
	\cite{diophantine_metabelian_grps}). Describing the set of solutions in any
	meaningful way has often proved difficult. The structure of sets of solutions
	in free groups were resolved by Razborov (\cite{Razborov_english},
	\cite{Razborov_thesis}), however the structure of solutions to systems of
	equations in soluble Baumslag-Solitar groups, and single equations in the
	Heisenberg group are some of the many cases that are yet to be described. The
	recent use of EDT0L languages in equations has helped to describe a number of
	cases.

	Section \ref{prelim_sec} covers the preliminaries of the topics used. In
	Section \ref{distinguished_letter_sec}, we prove Lemma
	\ref{intermesh_EDT0L_lem} on the parallel concatentation of words, which is an
	important part of the proofs of the stability of groups where systems of
	equations have EDT0L solution languages under direct products (Proposition
	\ref{dir_prod_EDT0L_prop}), and wreath products with finite groups
	(Proposition \ref{wreath_product_EDT0L_prop}). Section \ref{eqns_in_exts_sec}
	covers the proofs of those propositions, along with Lemma
	\ref{lem:virtual_sols}, which allows us to understand equations with rational
	constraints in finite index overgroups of certain groups. We then use Lemma
	\ref{lem:virtual_sols} to prove Theorem B in Section \ref{virt_abelian_sec},
	which shows that the solution language to a system of equations with rational
	contraints in a virtually abelian group has an EDT0L solution language.

	Section \ref{recognisable_constraints_sec} includes the addition of
	recognisable constraints to equations with EDT0L solutions, and is used to
	prove that the property of systems of equations having EDT0L solution
	languages passes to finite index subgroups, with respect to the Schreier
	normal form, based on the normal form used in the finite index overgroup
	(Proposition \ref{fin_index_EDT0L_prop}). Section \ref{virt_direct_prod_sec}
	concludes with the proof that systems of equations in groups that are
	virtually direct products of hyperbolic groups have EDT0L solution languages.

	\begin{ntn} We introduce some notation to be used throughout.
		\begin{itemize}
			\item Functions will be written to the right of their arguments.
			\item Let \(G\) be a group. We use \(\FIN(G)\) to denote the class of
			groups that contain \(G\) as a finite index subgroup.
			\item If \(S\) is a subset of a group, we define \(S^\pm = S \cup
			S^{-1}\).
			\item We use \(\varepsilon\) to denote the empty word.
			\item When defining endomorphisms of the free monoid \(\Sigma^\ast\), we
			do this by defining the action of the endomorphism on some of the elements
			of \(\Sigma\), and the remaining elements of \(\Sigma\) are assumed to be
			fixed. The endomorphism is determined by its action on \(\Sigma\).
			\item If \(L\) is a language over an alphabet \(\Sigma\), we use
			\(L^\text{c}\) to denote the complement of \(L\) within \(\Sigma^\ast\).
		\end{itemize}
	\end{ntn}

\section{Preliminaries}
	\label{prelim_sec}

	\subsection{Rational and recognisable sets}

	We cover here the basic definitions of rational and recognisable sets. Both
	types are used as constraints for variables in equations, and we will use
	recognisable constraints to show that the class of groups where solutions to
	systems of equations form EDT0L languages is closed under passing to finite
	index subgroups.

	Recall that a \textit{language} over \(\Sigma\) is any subset of
	\(\Sigma^\ast\), where \(\Sigma\) is a finite set, called an
	\textit{alphabet}. Recall also that a \textit{regular} language is any
	language accepted by a finite state automaton. We refer the reader to Chapter
	2 of \cite{groups_langs_aut} for further details on languages and finite state
	automata.

	\begin{dfn}
		Let \(G\) be a group, and \(\Sigma\) be a monoid generating set for \(G\).
		Define \(\pi \colon \Sigma^\ast \to G\) to be the natural homomorphism. We
		say a subset \(A \subseteq G\) is
		\begin{enumerate}
			\item \textit{recognisable} if \(A \pi^{-1}\) is a regular language over
			\(\Sigma\);
			\item \textit{rational} if there is a regular language \(L\) over
			\(\Sigma\), such that \(A = L \pi\).
		\end{enumerate}
	\end{dfn}

	\begin{rmk}
		Recognisable sets are rational.
	\end{rmk}

	We give a few examples of recognisable and rational sets.

	\begin{ex}
		Finite subsets of any group are rational. Finite subsets of a group \(G\)
		are recognisable if and only if \(G\) is finite \cite{Anisimov}. Finite
		index subgroups of any group are recognisable, and hence rational.
	\end{ex}

	The following result of Grunschlag relates the rational subsets
	of a finite index subgroup of a group \(G\) to the rational subsets of \(G\)
	itself.

	\begin{lem}[Corollary 2.3.8 in \cite{Grunschlag_thesis}]
		\label{Grunschlag_rational_lem}
		Let \(G\) be a group with finite generating set \(\Sigma\), and \(H\) be a
		finite index subgroup of \(G\). Let \(\Delta\) be a finite generating set
		for \(H\), and \(T\) be a right transversal for \(H\) in \(G\). For each
		rational subset \(R \subseteq G\), such that \(R \subseteq H t\) for some
		\(t \in T\), there exists a (computable) rational
		subset \(S \subseteq H\) (with respect to \(\Delta\)), such that \(R = St\).
	\end{lem}

	Herbst and Thomas proved that recognisable sets in a group \(G\) are always
	finite unions of cosets of a finite index normal subgroup of \(G\)
	\cite{herbst_thomas}. This can be used to prove many facts about recognisable
	sets, including the following lemma.

	\begin{lem}
		\label{fin_index_recognisable_lem}
		Let \(G\) be a finitely generated group with a finite index subgroup \(H\),
		and let \(S \subseteq H\). Then \(S\) is recognisable in \(G\) if and only
		if \(S\) is recognisable in \(H\).
	\end{lem}

	\subsection{Group equations}

	We define here a system of equations within a group, and certain
	generalisations including twisting and constraints. Twisted equations prove
	useful in proving that systems of equations with rational constraints in
	virtually abelian groups have EDT0L solutions.

	\begin{dfn}
		Let \(G\) be a group, and \(\mathcal X\) be a finite set of variables. A
		\textit{finite system of equations} in \(G\) with \textit{variables}
		\(\mathcal X\) is a finite subset \(\mathcal E\) of \(G \ast
		F_\mathcal{X}\), where \(F_\mathcal{X}\) is the free group on a finite set
		\(\mathcal{X}\). If \(\mathcal E = \{w_1, \ \ldots, \ w_n\}\), we view
		\(\mathcal E\) as a system by writing \(w_1 = w_2 = \cdots = w_n = 1\). A
		\textit{solution} to a system \(w_1 = \cdots = w_n = 1\) is a homomorphism
		\(\phi \colon F_{\mathcal X} \to G\), and such that \(w_1 \bar{\phi} =
		\cdots = w_n \bar{\phi} = 1_G\), where \(\bar{\phi}\) is the extension of
		\(\phi\) to a homomorphism from \(G \ast F_\mathcal{X} \to G\), defined by
		\(g \bar{\phi} = g\) for all \(g \in G\).


		Let \(\Omega \leq \Aut(G)\). A \textit{finite system of \(\Omega\)-twisted
		equations} in \(G\) with \textit{variables} \(\mathcal X\) is a finite
		subset \(\mathcal E\) of the monoid \((G \cup F_\mathcal{X} \times
		\Omega)^\ast\), and is again denoted \(w_1 = \cdots = w_n = 1\). Define the
		function
		\begin{align*}
			p \colon G \times \Aut(G) & \to G \\
			(g, \ \psi) & \mapsto g \psi.
		\end{align*}
		If \(\phi \colon F_\mathcal{X} \to G\) is a homomorphism, let \(\bar{\phi}\)
		denote the (monoid) homomorphism from \((G \cup F_\mathcal{X} \times
		\Omega)^\ast\) to \((G \times \Omega)^\ast\), defined by \((h, \ \psi)
		\bar{\phi} = (h \phi, \ \psi)\) for \((h, \ \psi) \in F_\mathcal{X} \times
		\Omega\) and \(g \bar{\phi} = g\) for all \(g \in G\). A \textit{solution}
		is a homomorphism \(\phi \colon F_\mathcal{X} \to G\), such that \(w_1
		\bar{\phi} p = \cdots = w_n \bar{\phi} p = 1_G\). When \(\Omega = \Aut(G)\),
		we omit the reference to \(\Omega\), and call such a system a \textit{finite
		system of twisted equations}.

		For the purposes of decidability, in finitely generated groups, the elements
		of \(G\) will be represented as words over a finite generating set, and in
		twisted equations, automorphisms will be represented by their action on the
		generators.

		A \textit{finite system of (twisted) equations with rational (recognisable)
		constraints} \(\mathcal{E}\) in a group \(G\) is a finite system of
		(twisted) equations \(\mathcal{F}\) with variables \(X_1, \ \ldots, \ X_n\),
		together with a tuple of rational (recognisable) subsets \(R_1, \ \ldots, \
		R_n\) of \(G\). A \textit{solution} to \(\mathcal{E}\) is a solution
		\(\phi\) to \(\mathcal{F}\), such that \(X_i \phi \in R_i\) for all \(i\).
	\end{dfn}

	\begin{rmk}
		A solution to an equation with variables \(X_1, \ \ldots, \ X_n\) will
		usually be represented as a tuple \((x_1, \ \ldots, \ x_n)\) of group
		elements, rather than a homomorphism. We can obtain the homomorphism from
		the tuple by defining \(X_i \mapsto x_i\) for each \(i\).
	\end{rmk}

	\begin{ex}
		Equations in \(\mathbb{Z}\) are linear equations in integers, and elementary
		linear algebra can be used to determine satisfiability, and also describe
		solutions.
	\end{ex}

	\begin{ex}
		The conjugacy problem in any group can be viewed as an equation \(X^{-1} g X
		= h\), where \(g\) and \(h\) are group elements, and \(X\) is a variable.
		For example, in the free group \(F(a, \ b)\), one could consider the
		equation \(X^{-1} ab X = ba\). The set of solutions is \(\{(ab)^n b^{-1}
		\mid n \in \mathbb{Z}\}\).

		The twisted conjugacy problem can similarly be viewed, using the equation
		\(X^{-1} g X = h \Phi\), for some automorphism \(\Phi\).
	\end{ex}

	\begin{ex}
		Let \(\Phi = \left(\begin{array}{cc} 0 & 1 \\ -1 & 0	 \end{array} \right)
		\in \GL_2(\mathbb{Z})\). Consider the twisted equation in \(\mathbb{Z}^2\),
		with the variables \(\mathbf X\) and \(\mathbf Y\):
		\[
			(\mathbf{X}) \Phi = \mathbf{Y}.
		\]
		This is just the automorphism problem in \(\mathbb{Z}^2\), which can be
		solved using elementary linear algebra.
		In the free group \(F(a, \ b)\), an example of a twisted equation would be
		\(X (Y \Psi) a Y = b X^{-1}\), with \(\Psi \in \Aut(F(a, \ b))\) defined by
    \(a \Psi = ab\) and \(b \Psi = b\). Computing solutions to this is
		more difficult, although an algorithm does exist to construct the set of
    solutions (see \cite{VF_eqns}).
	\end{ex}

		\subsection{Space complexity}

			We briefly define space complexity. We refer the reader to
			\cite{computational_compl} for a comprehensive introduction to space
			complexity, or to \cite{eqns_free_grps} for the consideration of space
			complexity when constructing EDT0L systems.

			\begin{dfn}
				Let \(f \colon \mathbb{Z}_{\geq 0} \to \mathbb{Z}_{\geq 0}\). An
				algorithm is said to run in \(\mathsf{NSPACE}(f)\)
				(\textit{non-deterministic \(f\) space}) if it can be performed by a
				non-deterministic Turing machine with the following:
				\begin{enumerate}
					\item A read-only input tape;
					\item A write-only output tape;
					\item A read-write work tape such that no computation path in the
					Turing machine uses	more than \(\mathcal{O}(nf)\) units of the work
					tape, for an input of length \(n\).
				\end{enumerate}
			\end{dfn}

			We will use space complexity to show that the EDT0L systems we will
			construct, can be done so with a clear bound on the amount of memory used.
			We will need the following standard result later on.

			\begin{lem}
				\label{reg_lang_closure_prop_lem}
				Let \(L\) and \(M\) be regular languages over alphabets \(\Sigma_L\) and
				\(\Sigma_M\), that are constructible in \(\ns(f)\) for some \(f \colon
				\mathbb{Z}_{\geq 0} \to \mathbb{Z}_{\geq 0}\). Let \(\phi \colon
				\Sigma_L^\ast \to \Sigma_M^\ast\) be a free monoid homomorphism. Then
				the following languages are regular, and constructible in \(\ns(f)\).
				\begin{enumerate}
					\item \(L \cup M\) (union);
					\item \(L \cap M\) (intersection);
					\item \(L \phi\) (homomorphism);
					\item \(M \phi^{-1}\) (inverse homomorphism).
				\end{enumerate}
			\end{lem}

			\begin{proof}
				Let \(\mathcal A_L\) and \(\mathcal A_M\) be finite state automata
				accepting \(L\) and \(M\), respectively, that are both constructible in
				\(\ns(f)\). Let \(q_L\) and \(F_L\), and \(q_M\) and \(F_N\) be the
				start states, and sets of accept states of \(\mathcal A_L\) and
				\(\mathcal A_M\), respectively.
				\begin{enumerate}
					\item The finite state automaton that accepts \(L \cup M\) is obtained
					by taking the union \(\mathcal A_L \cup \mathcal A_M\), and adding an
					additional state, \(q_0\). We attach an \(\varepsilon\)-labelled edge
					from \(q_0\) to \(q_L\) and \(q_M\), and then set \(q_0\) to be the
					start state. The accept states will be \(F_L \cup F_M\). Printing this
					can be done using the memory required to print both of \(\mathcal
					A_L\) and \(\mathcal A_M\), plus a constant, and thus it is
					constructible in \(\ns(f)\).
          \item We can take \(\mathcal A_L \times \mathcal A_M\), to be our
            finite state automaton for \(L \cap M\), where the start state is
            \(q_L \times q_M\), and the set of accept states is \(F_L \times
            F_M\), and with additional \(\varepsilon\)-transitions from a state
            \((p_L, \ p_M)\) to \((p_L, \ p_M')\) whenever there is a
            \(\varepsilon\)-transition from \(p_M\) to \(p_M'\) in \(\mathcal
            A_M\) and the analogous \(\varepsilon\)-transitions for each
            \(\varepsilon\)-transition in \(\mathcal A_L\). To write this down,
            we proceed with the construction of \(\mathcal A_L\), but whenever
            we would normally output a state \(q\), we instead output \(\{q\}
            \times \mathcal A_M\), and whenever we would add an edge between
            states \(q_1\) and \(q_2\), we instead add all edges between
            \(\{q_1\} \times \mathcal A_M\) and \(\{q_2\} \times \mathcal
            A_M\), by going through the construction of \(\mathcal A_M\).  To
            do this, we never need to store more than the information required
            to write down both \(\mathcal A_M\) and \(\mathcal A_L\) plus a
            constant, and thus this can be completed in \(\ns(f)\).
					\item We can do this by constructing \(\mathcal A_L\), except whenever
					we would output an edge labelled with \(a\), we instead output a path
					labelled with \(a \phi\).
					\item We adapt the construction in Proposition 3.3 in
					\cite{eqns_hyp_grps}. First let \(\bar{\Sigma}_M = \{\bar{a} \mid a
					\in \Sigma_M\}\) be a copy of \(\Sigma_M\), disjoint with
					\(\Sigma_L\), and let \(\bar{M}\) be the language obtained from \(M\)
					by replacing every occurence of \(a \in \Sigma_M\) with \(\bar{a}\).
					Now let
					\[
						K = \{y_0 \bar{x}_1 y_1 \cdots \bar{x}_n y_n \mid n \in
						\mathbb{Z}_{> 0}, \ \bar{x}_1 \cdots \bar{x}_n \in \bar{M}, \ y_1, \
						\ldots, \ y_n \in \Sigma_L^\ast\}.
					\]
					We can construct a finite state automaton accepting \(K\) in
					\(\ns(f)\), by constructing \(\mathcal A_M\), but replacing each
					occurence of \(a \in \Sigma_M\) with \(\bar{a}\), and then for each
					\(b \in \Sigma_L\), adding a loop in each vertex labelled with \(b\).
					Now consider the regular language
					\[
						S = \{(y_1 \phi) \bar{y}_1 (y_2 \phi) \bar{y}_2 \cdots (y_n \phi)
						\bar{y}_n \mid n \in \mathbb{Z}_{> 0}, \ y_1, \ \ldots, \ y_n \in
						\Sigma_L^\ast\}.
					\]
					Note that the size of \(S\) is constant; it depends only on \(\phi\).
					Let \(\tau \colon (\Sigma_L \cup \bar{\Sigma}_M)^\ast \to
					(\Sigma_L)^\ast\) be the free monoid homomorphism defined by \(a \tau
					= a\) if \(a \in \Sigma_L\), and \(\bar{a} \tau = \varepsilon\), if
					\(\bar{a} \in \bar{\Sigma}_M\). By construction, \(M \phi^{-1} = (K
					\cap S) \tau\). Using (2) and (3), it follows that \(M \phi^{-1}\) is
					constructible in \(\ns(f)\).
				\end{enumerate}
			\end{proof}

	\subsection{EDT0L languages}

	The class of languages we use to describe solution sets is the class of EDT0L
	languages. All EDT0L languages are ET0L, which are indexed languages, and
	hence context-sensitive, and all regular languages are EDT0L. However, there
	are EDT0L languages that are not context-free, and context-free languages that
	are not EDT0L \cite{CF_not_EDT0L}. Context-free languages do not work
	naturally with systems of equations, as most equations with \(3\) or more
	variables will not have context-free solutions; the system \(XY^{-1} = XZ^{-1}
	= 1\) in \(\mathbb{Z}\), using the presentation \(\langle a | \rangle\), will
	have the solution language \(\{a^m \# a^m \# a^m \mid m \in \mathbb{Z}\}\),
	which is not context-free. For more information on EDT0L languages, we refer
	the reader to \cite{math_theory_L_systems} and \cite{handbook_form_lang}.

	We can now define EDT0L languages. We base our definitions on \cite{VAEP},
	however there are a number of equivalent definitions used elsewhere.

	\begin{dfn}
		An \textit{EDT0L system} is a tuple \(\mathcal H = (\Sigma, \ C, \ \omega, \
		\mathcal{R})\), where
		\begin{enumerate}
			\item \(\Sigma\) is an alphabet, called the \textit{(terminal) alphabet};
			\item \(C\) is a finite superset of \(\Sigma\), called the
			\textit{extended alphabet} of \(\mathcal H\);
			\item \(\omega \in C^\ast\) is called the \textit{start word};
			\item \(\mathcal{R}\) is a regular (as a language) set of endomorphisms of
			\(C^\ast\), called the \textit{rational control} of \(\mathcal H\).
		\end{enumerate}
		The language \textit{accepted} by \(\mathcal H\) is \(L(\mathcal H) =
		\{\omega \phi \mid \phi \in \mathcal{R}\} \cap \Sigma^\ast\).



		A language that is accepted by some EDT0L system is called an
		\textit{EDT0L language}.
	\end{dfn}

	The following is a standard example of an EDT0L language that is not
	context-free.

	\begin{ex}
    \label{n2_EDT0L_ex}
    The language \(L = \{a^{n^2} \mid n \in \mathbb{Z}_{> 0}\}\) is an EDT0L
    language over the alphabet \(\{a\}\). This can be seen by considering
    \(\omega = \perp\}\) and \(C = \{\perp, \ s, \ t, \ u, \ a\}\) as the
    extended alphabet and start word, respectively, of an EDT0L system accepting
    \(L\), and using the finite state automaton from Figure \ref{n2_EDT0L_fig}
    to define the rational control.
    \begin{figure}
			\caption{Rational control for \(L = \{a^{n^2} \mid n \in
			\mathbb{Z}_{>0}\}\), with start state \(q_0\) and accept state \(q_3\).}
			\label{n2_EDT0L_fig}
      \begin{tikzpicture}
        [scale=.8, auto=left,every node/.style={circle}]
        \tikzset{
        on each segment/.style={
          decorate,
          decoration={
            show path construction,
            moveto code={},
            lineto code={
              \path [#1]
              (\tikzinputsegmentfirst) -- (\tikzinputsegmentlast);
            },
            curveto code={
              \path [#1] (\tikzinputsegmentfirst)
              .. controls
              (\tikzinputsegmentsupporta) and (\tikzinputsegmentsupportb)
              ..
              (\tikzinputsegmentlast);
            },
            closepath code={
              \path [#1]
              (\tikzinputsegmentfirst) -- (\tikzinputsegmentlast);
            },
          },
        },
        mid arrow/.style={postaction={decorate,decoration={
              markings,
              mark=at position .5 with {\arrow[#1]{stealth}}
            }}},
      }

        \node[draw, initial] (q0) at (0, 0) {\(q_0\)};
        \node[draw] (q1) at (0, -5)  {\(q_1\)};
        \node[draw] (q2) at (5, -5) {\(q_2\)};
        \node[draw, double] (q3) at (0, -10) {\(q_3\)};

        \draw[postaction={on each segment={mid arrow}}] (q0) to (q1);

        \draw[postaction={on each segment={mid arrow}}] (q1) to
        [out=40, in=140, distance=1cm] (q2);

        \draw[postaction={on each segment={mid arrow}}] (q2) to
        [out=-140, in=-40, distance=1cm] (q1);

        \draw[postaction={on each segment={mid arrow}}] (q1) to (q3);

        \node (l1) at (-1.7, -2.35) {\(\varphi_\perp
        \colon \perp \mapsto tsa\)};

        \node (l2) at (2.5, -3.6) {\(\varphi_1 \colon s \mapsto su\)};

        \node (l3a) at (2.5, -6.3) {\(\varphi_2 \colon t \mapsto at\)};

        \node (l3b) at (3.05, -6.8) {\(u \mapsto ua^2\)};

        \node (l4) at (-1.8, -7.35) {\(\varphi_3 \colon
        s, t, u \mapsto \varepsilon\)};

      \end{tikzpicture}
    \end{figure}
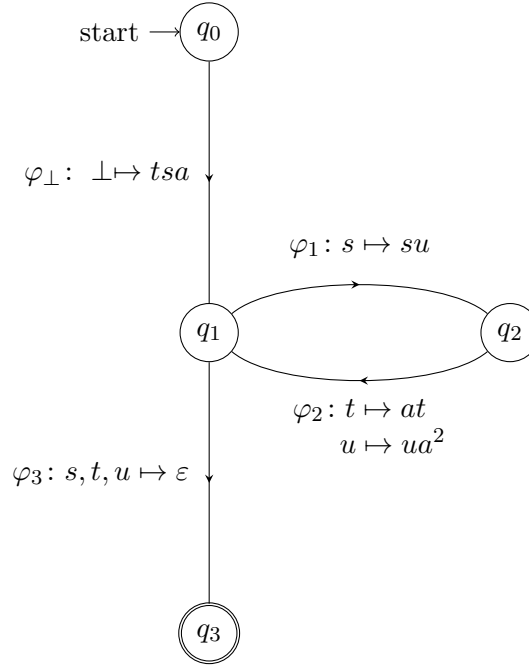
    Note that the rational control can also be written as \(\varphi_\perp
    (\varphi_1 \varphi_2)^\ast \varphi_3\).

		This language is not context-free. This can be shown using the pumping
		lemma (Theorem 2.6.17 in \cite{groups_langs_aut}).

  \end{ex}

	An \textit{abstract family of languages} is one closed under the five
	operations in the following lemma, together with preimages under free monoid
	homomorphisms. The following lemma shows that even if EDT0L languages do not
	form a full abstract family of languages (Theorem V.2.17 in
	\cite{math_theory_L_systems}) like regular (Proposition 2.5.10 and Proposition
	2.5.14 in \cite{groups_langs_aut}), context-free (Proposition 2.6.27,
	Proposition 2.6.31 and Proposition 2.6.34 in \cite{groups_langs_aut}) and ET0L
	languages (Theorem V.1.7 in \cite{math_theory_L_systems}), they are closed
	under most of the standard operations that are frequently used to manipulate
	languages.

	\begin{lem}
    \label{EDT0L_closure_properties_lem}
    The class of EDT0L languages is closed under the following operations:
    \begin{enumerate}
      \item Finite unions;
      \item Intersection with regular languages;
      \item Concatenation;
      \item Kleene star closure;
      \item Image under free monoid homomorphisms.
    \end{enumerate}
		Moreover, if the EDT0L systems used in any of these operations can be
		constructed in \(\mathsf{NSPACE}(f)\), for some \(f \colon \mathbb{Z}_{\geq
		0} \to \mathbb{Z}_{\geq 0}\), then there is a computable EDT0L system
		accepting the resultant language that can also be constructed in
		\(\mathsf{NSPACE}(f)\).
  \end{lem}

	\begin{proof}
		The proofs of (1), (2) and (5) can be found in Proposition 3.3 in
		\cite{eqns_hyp_grps}. We now show (3) and (4). Let \(f \colon
		\mathbb{Z}_{\geq 0} \to \mathbb{Z}_{\geq 0}\). Let \((\Sigma_L, \ C_L, \
		\omega_L, \ \mathcal{R}_L)\) and \((\Sigma_M, \ C_M, \ \omega_M, \
		\mathcal{R}_M)\) be EDT0L systems that are constructible in \(\ns(f)\),
		accepting languages \(L\) and \(M\), respectively.
		\begin{enumerate}
			\item[(3)] Let \(\bar{\Sigma}_M = \{\bar{a} \mid a \in M\}\) be a copy of
			\(\Sigma_M\) that is disjoint from \(\Sigma_L\), and let \(\bar{M} =
			\{\bar{w} \mid w \in M\}\). By modifying the EDT0L system for \(M\) to
			replace occurences of unbarred letters with barred ones, we have that
			there exists an EDT0L system \((\bar{\Sigma}_M, \ C_{\bar{M}}, \
			\omega_{\bar{M}}, \ \mathcal{R}_{\bar{M}})\) accepting \(\bar{M}\), that
			is constructible in \(\ns(f)\). We can assume without loss of generality
			that \(C_L \backslash \Sigma_L\) and \(C_{\bar{M}} \backslash
			\bar{\Sigma}_M\) are disjoint (and thus \(C_L\) and \(C_{\bar{M}}\) are
			disjoint).

			As \(\mathcal{R}_L\) is regular, there is a finite set \(B_L \subseteq
			\End(C_L^\ast)\) over which \(\mathcal{R}_L\) is a regular language. For
			each \(\psi \in B_L\), define \(\hat{\psi} \in \End((C_L \cup
			C_{\bar{M}})^\ast)\) by
			\[
				a \hat{\psi} = \left\{
				\begin{array}{cl}
					a \psi & a \in C_L \\
					a & a \in C_{\bar{M}}.
				\end{array}
				\right.
			\]
			Similarly define \(\hat{\psi} \in \End((C_L \cup C_{\bar{M}})^\ast)\) for
			each \(\psi\) in the finite set over which \(\mathcal{R}_{\bar{M}}\) is a
			regular language. Let \(\hat{\mathcal{R}}_L\) and \(\hat{\mathcal{R}}_M\)
			be the regular languages obtained from \(\mathcal{R}_L\) and
			\(\mathcal{R}_{\bar{M}}\) by replacing each \(\psi \in B_L\) (or the
			equivalent for \(\bar{M})\) with their hatted versions.

			By construction, the concatentation \(L \bar{M}\) is accepted by the EDT0L
			system \(\mathcal{H} = (\Sigma_L \cup \bar{\Sigma}_M, \ C_L \cup C_M, \
			\omega_L \omega_{\bar{M}}, \hat{\mathcal{R}}_L \cup
			\hat{\mathcal{R}}_M)\). Since \(\mathcal{R}_L\) and
			\(\mathcal{R}_{\bar{M}}\) are constructible in \(\ns(f)\), it follows that
			\(\hat{\mathcal{R}}_L\) and \(\hat{\mathcal{R}}_{\bar{M}}\) are as well.
			Thus all of the sets used to define \(\mathcal{H}\) are constructible in
			\(\ns(f)\), and so \(\mathcal{H}\) is. We can conclude that \(L \bar{M}\)
			is accepted by an EDT0L system that is constructible in \(\ns(f)\).

			Let \(\theta \colon (\Sigma_L \cup \bar{\Sigma}_M^\ast) \to \Sigma_M\) be
			the free monoid homomorphism defined by \(a \theta = a\) if \(a \in
			\Sigma_L\) and \(a \theta = b\) if \(a = \bar{b} \in \bar{\Sigma}_M\). It
			follows that \((L \bar{M}) \theta = L M\), and so \(LM\) is accepted by an
			EDT0L system that is constructible in \(\ns(f)\) by part (5).
			\item[(4)] Let \(\bar{\Sigma}_L = \{\bar{a} \mid a \in \Sigma_L\}\) be a
			disjoint copy of \(\Sigma_L\). Let \(\bar{L} = \{\bar{w} \mid w \in L\}\).
			Note that there is an EDT0L system \((\bar{\Sigma}_L, \ C_{\bar{L}}, \
			\omega_{\bar{L}}, \ \mathcal{R}_{\bar{L}})\) accepting \(\bar{L}\), that
			is constructible in \(\ns(f)\). Let \(\theta \colon \bar{\Sigma}_L^\ast
			\to \Sigma_L^\ast\) be defined by \(\bar{a} \theta = a\). Thus \(L\)
			is accepted by the EDT0L system \(\mathcal{H} = (\Sigma_L, \ C_L \cup
			\bar{\Sigma}_L, \ \omega_{\bar{L}}, \ \mathcal{R}_{\bar{L}} \theta)\)
			(where endomorphisms in \(\mathcal{R}_{\bar{L}}\) have been extended to
			act on \((\Sigma_L \cup C_{\bar{L}})^\ast\) as the identity function on
			\(\Sigma_L\)). We have
			that endomorphisms of the finite set over which the rational control of
			\(\mathcal{H}\) is a regular language fix \(\Sigma_L\). So we can assume
			without loss of generality that this is true in \((\Sigma_L, \ C_L, \
			\omega_L, \ \mathcal{R}_L)\).

			Let \(\perp \notin C_L\). Define endomorphisms \(\sigma\) and \(\psi\)
			over \(C \cup \{\perp\}\) by \(\perp \sigma = \{\omega \perp\}\) and
			\(\perp \psi = \{\varepsilon\}\).  By construction, the EDT0L system
			\((\Sigma_L, \ C_L \cup \{\perp\}, \ \perp, \ (\sigma \mathcal{R}_L)^\ast
			\psi)\) (where endomorphisms in \(\mathcal{R}_L\) have been extended to
			act on \((C_L \cup \{\perp\})^\ast\) as the identity function on
			\(\perp\)). Since \(\psi\) and \(\sigma\) can be constructed in constant
			space, this EDT0L system can be constructed in \(\ns(f)\).
		\end{enumerate}
	\end{proof}

	\subsection{Solution languages}
	We now explain how we represent solution sets as languages. We start by
	defining a normal form.

	\begin{dfn}
		Let \(G\) be a group, and \(\Sigma\) be a finite generating set for \(G\). A
		\textit{normal form} for \(G\), with respect to \(\Sigma\) (alternatively, a
		normal form for \((G, \ \Sigma)\)), is a function \(\eta \colon G \to
		(\Sigma^\pm)^\ast\) such that \(g \eta\) represents \(g\) for all \(g \in
		G\).

		A normal form \(\eta\) is called
		\begin{enumerate}
			\item \textit{regular} if \(\im \eta\) is a regular language over
			\(\Sigma^\pm\);
			\item \textit{geodesic} if \(\im \eta\) comprises only geodesic words
			in \(G\), with respect to \(\Sigma\);
			\item \textit{quasigeodesic} if there exists \(\lambda > 0\) such that
			\(|g \eta| \leq \lambda |g|_{(G, \Sigma)} + \lambda\) for all \(g \in G\).
		\end{enumerate}
	\end{dfn}

	Note that we are insisting our normal forms produce a unique representative
	for each element.

	We are now in a position to represent solutions as languages, with respect to
	a given normal form.

	\begin{dfn}
		Let \(G\) be a group with a finite monoid generating set \(\Sigma\), and
		let \(\eta \colon G \to (\Sigma^\pm)^\ast\) be a normal form for \(G\) with
		respect to \(\Sigma\). Let \(\mathcal{E}\) be a system of equations in \(G\)
		with a set \(\mathcal{S}\) of solutions. The \textit{solution language}
		to \(\mathcal{E}\), with respect to \(\eta\), is the language
		\[
			\{(g_1 \eta) \# \cdots \# (g_n \eta) \mid (g_1, \ \ldots, \ g_n) \in
			\mathcal{S}\}
		\]
		over \(\Sigma^\pm \sqcup \{\#\}\).
	\end{dfn}

	We give an example of an equation in a group, with an EDT0L language of
	solutions.

	\begin{ex}
		Consider the equation \(XY^{-1} = 1\) in \(\mathbb{Z}\) with the
		presentation \(\langle a \mid \rangle\). The solution language with
		respect to the standard normal form is
		\[
			L = \{a^n \# a^n \mid n \in \mathbb{Z}\},
		\]
		over the alphabet \(\{a, \ a^{-1}, \ \#\}\). The language \(L\) is EDT0L;
		our system will have the extended alphabet \(\{\perp, \ \#, \ a, \
		a^{-1}\}\), start word \(\perp \# \perp\), and rational control defined by
		Figure \ref{a^na^n_fig}.
		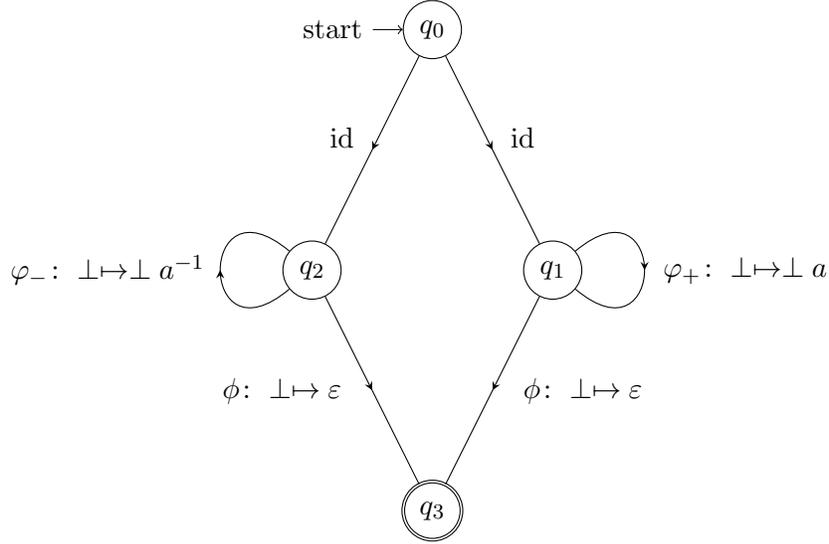
\begin{figure}
			\caption{Rational control for \(L = \{a^n \# a^n \mid n \in \mathbb{Z}\}\)
			with start state \(q_0\), and accept state \(q_3\).}
			\label{a^na^n_fig}
			\begin{tikzpicture}
				[scale=.8, auto=left,every node/.style={circle}]
				\tikzset{
				on each segment/.style={
					decorate,
					decoration={
						show path construction,
						moveto code={},
						lineto code={
							\path [#1]
							(\tikzinputsegmentfirst) -- (\tikzinputsegmentlast);
						},
						curveto code={
							\path [#1] (\tikzinputsegmentfirst)
							.. controls
							(\tikzinputsegmentsupporta) and (\tikzinputsegmentsupportb)
							..
							(\tikzinputsegmentlast);
						},
						closepath code={
							\path [#1]
							(\tikzinputsegmentfirst) -- (\tikzinputsegmentlast);
						},
					},
				},
				mid arrow/.style={postaction={decorate,decoration={
							markings,
							mark=at position .5 with {\arrow[#1]{stealth}}
						}}},
			}

				\node[draw, initial] (q0) at (0, 0) {\(q_0\)};
				\node[draw] (q1) at (2, -4)  {\(q_1\)};
				\node[draw] (q2) at (-2, -4) {\(q_2\)};
				\node[draw, double] (q3) at (0, -8) {\(q_3\)};

				\draw[postaction={on each segment={mid arrow}}] (q0) to (q1);
				\draw[postaction={on each segment={mid arrow}}] (q0) to (q2);

				\draw[postaction={on each segment={mid arrow}}]
				(q1) to [out=40, in=-40, distance=2cm] (q1);

				\draw[postaction={on each segment={mid arrow}}] (q2) to
				[out=-140, in=140, distance=2cm] (q2);

				\draw[postaction={on each segment={mid arrow}}] (q1) to (q3);
				\draw[postaction={on each segment={mid arrow}}] (q2) to (q3);

				\node (l1) at (1.5, -1.8) {\(\id\)};
				\node (l2) at (-1.5, -1.8) {\(\id\)};
				\node (l3) at (5.2, -4) {\(\varphi_+ \colon \perp \mapsto \perp a\)};
				\node (l4) at (-5.4, -4) {\(\varphi_- \colon \perp \mapsto \perp
				a^{-1}\)};
				\node (l5) at (2.5, -6) {\(\phi \colon \perp \mapsto \varepsilon\)};
				\node (l5) at (-2.5, -6) {\(\phi \colon \perp \mapsto \varepsilon\)};

			\end{tikzpicture}
		\end{figure}
		Note that \(\id\) denotes the identity function, and the rational control
		can also be expressed using the rational expression \(\{\varphi_-^\ast, \
		\varphi_+^\ast\} \phi\).
	\end{ex}

	\begin{rmk}
		We now introduce space complexity to solution languages. We first need to
		define the `size' of a system of equations, in order to measure our input.
		The definition of size can vary, as specific groups can have different ways
		of writing equations. For example, in \cite{VAEP}, equations in virtually
		abelian groups were stored as tuples of integers, as this compressed the
		size of the equations, whilst storing all of the necessary information. This
		approach has not always been used in other cases when compression was
		possible. Since we deal with virtually abelian groups on their own, we will
		use this definition when referring to equations in virtually abelian groups.

		When discussing equations in constructions based on other groups (such as
		direct products, finite index subgroups, wreath products) we will `inherit'
		the input definition from the groups they are defined from. If these vary,
		we will use the general definition, which is less efficient than the
		specific virtually abelian case, and as a result, will still yield (at
		least) the same results. The general definition of equation size will also
		be used when talking about groups that are virtually direct products of
		hyperbolic groups.
	\end{rmk}

	We start with the general definition of equation length.

	\begin{dfn}
		Let \(G\) be a group, and \(\omega = 1\) be an equation in \(G\). Recall
		that \(\omega \in F_V \ast G\), for some finite set \(V\). Fix a generating
		set \(\Sigma\) for \(G\).  We define the \textit{length} of \(\omega = 1\)
		to be the length of the group element \(\omega \in F_V \ast G\), with
		respect to the generating set \(\Sigma \cup V\).

		Let \(\mathcal E\) be a finite system of equations in \(G\). The
		\textit{length} of \(\mathcal E\) is the sum of the lengths of all equations
		in \(\mathcal E\).
	\end{dfn}

	Before we define virtually abelian equation length, we must first consider the
	free abelian case. The compression is possible because we can store an
	integer \(n\) with \(\log n + c\) bits, for some constant \(c\). This is
	covered in greater detail in Remark 3.6 and Remark 3.10 in \cite{VAEP}.

	\begin{dfn}
		Let \(a_1, \ \ldots, \ a_k\) denote the standard generators of
		\(\mathbb{Z}^k\). Let \(\omega = 1\) be an equation in \(\mathbb{Z}^k\)
		with a set \(\{X_1, \ \ldots, \ X_n\}\) of variables. By reordering a given
		expression for \(\omega\), we can assume \(\omega =  1\) is in the form
		\[
			X_1^{b_1} \cdots X_n^{b_n} a_1^{c_1} \cdots a_k^{c_k} = 1,
		\]
		where \(b_1, \ \ldots, \ b_n, \ c_1, \ \ldots, \ c_k \in \mathbb{Z}
		\backslash \{0\}\) (in the case when these values equal zero, we simply omit
		the relevant variables or generators from the equation). We can then define
		the \textit{free abelian length} of \(\omega = 1\) to be
		\[
			\sum_{i = 1}^n \log|b_i| + \sum_{j = 1}^k \log|c_j| + Ckn.
		\]
		Suppose now \(\nu = 1\) is a twisted equation in \(\mathbb{Z}^k\). By
		rearranging \(\nu\), we can assume it is of the form
		\[
			(X_1 B_1) \cdots (X_n B_n) a_1^{c_1} \cdots a_k^{c_k} = 1,
		\]
		where each \(B_r = [b_{rij}]\) is a \(k \times k\) integer-valued matrix
		(not-necessarily invertible). These are described in more detail in the
		proof of Lemma 3.3 in \cite{VAEP}. The \textit{free abelian length} of
		\(\nu = 1\) is defined to be
		\[
			\sum_{r = 1}^n \sum_{i = 1}^k \sum_{j = 1}^k \log|b_{rij}| + C' nk^2 +
			\sum_{j = 1}^k \log|c_j| + C'k.
		\]
 		where \(C'\) is a constant.

		From \cite{VAEP}, any equation \(\xi = 1\) in a virtually abelian group
		induces a twisted equation \(\bar{\xi} = 1\) in a free abelian group,
		which is unique up to the choice of transversal. We fix a choice of
		transversal, then define the \textit{virtually abelian length} of
		\(\xi = 1\) to be the free abelian length of \(\bar{\xi} = 1\).

		Let \(\mathcal E\) be a finite system of equations in a virtually abelian
		group. The \textit{virtually abelian length} of \(\mathcal E\) is the sum of
		the virtually abelian lengths of all equations in \(\mathcal E\).
		\textit{Free abelian length} of a system of equations is defined
		analogously.
	\end{dfn}

	We now use these lengths as our input size. Unless we explicitly state that we
	are using virtually or free abelian equation length, we will assume we are
	using the general version of equation length.


	\begin{dfn}
		Let \(\mathcal C\) be a class of languages, and fix a type of machine or
		grammar that constructs languages in \(\mathcal C\). Let \(f \colon
		\mathbb{Z}_{\geq 0} \to \mathbb{Z}_{\geq 0}\). Let \(G\) be a group with
		a finite generating set \(\Sigma\), and let \(\eta\) be a normal form for
		\((G, \ \Sigma)\). We say that solutions to systems of equations in \(G\),
		with respect to \(\eta\), are \textit{\(\mathcal C\) in \(\ns(f)\)} if
		\begin{enumerate}
			\item The solution languages to systems of equations in \(G\) are
			\(\mathcal C\) with respect to \(\eta\);
			\item Given a system of equations \(\mathcal E\) in \(G\), a machine or
			grammar accepting the solution language can be constructed in \(\ns(f)\),
			with \(\mathcal E\) as the input.
		\end{enumerate}
	\end{dfn}

	\begin{rmk}
		Since the only class of languages we will be using to describe solutions is
		the class of EDT0L languages, we will only be saying \textit{EDT0L in
		\(\ns(f)\)}, and the type of grammar we refer to when we say this is the
		EDT0L system.
	\end{rmk}

	\subsection{Dihedral Artin groups}

	We briefly define dihedral Artin groups. An application of Corollary
	\ref{main_thm_cor} is that solution sets to systems of equations in these
	groups form EDT0L languages.

	\begin{dfn}
		A \textit{dihedral Artin group} \(\text{DA}_m\), where \(m \geq 2\), is
		defined by the presentation
		\[
			\langle a, \ b \mid \underbrace{aba \cdots}_m = \underbrace{bab \cdots}_m
			\ \rangle.
		\]
	\end{dfn}

	The following lemma is widely known. A brief sketch of the proof can be found
	in Section 2 of \cite{equations_VDP}.

	\begin{lem}
		\label{dihedral_artin_virt_dir_prod_lem}
		A dihedral Artin group is virtually a direct product of free groups.
	\end{lem}

	\subsection{Schreier generators}
	We use Schreier generators, along with the normal form they induce, in order
	to show that the class of groups where systems of equations have EDT0L
	languages of solutions is stable under passing to finite index subgroups. This
	subsection is based on Section 1.4 of \cite{groups_langs_aut}.

	We start with the definition of Schreier generators.

	\begin{dfn}
		Let \(G\) be a group, generated by a finite set \(\Sigma\), \(H\) be a
		finite index subgroup of \(G\), and \(T\) be a right transversal of \(H\) in
		\(G\). For each \(g \in G\), let \(\bar{g}\) be the (unique) element of
		\(T\) that lies in the coset \(Hg\). The \textit{Schreier generating set}
		for \(H\), with respect to \(T\) and \(\Sigma\), is defined to be
		\[
			Z = \{tx \overline{tx}^{-1} \mid t \in T, \ x \in \Sigma\}.
		\]
	\end{dfn}

	Whilst the fact that the Schreier generating set is a finite generating set
	for \(H\) is widely known, we include a proof, as we later
	use ideas from the proof.
	%


	%

	\begin{lem}
		\label{schreier_lem}
		Let \(G\) be a group, generated by a finite set \(\Sigma\), \(H\) be a
		finite index subgroup of \(G\), and \(T\) be a right transversal of \(H\) in
		\(G\). Let \(Z\) be the Schreier generating set for \(H\). Then \(Z\) is
		finite, and \(H = \langle Z \rangle\).
	\end{lem}

	\begin{proof}
		We first show that
		\[
			Z^{-1} = \{tx^{-1} \overline{tx^{-1}}^{-1} \mid t \in T, \ x \in \Sigma\}.
		\]
		Let \(S = \{tx^{-1} \overline{tx^{-1}}^{-1} \mid t \in T, \ x \in
		\Sigma\}\). Let \(g = \overline{tx}x^{-1} t^{-1} = (tx
		\overline{tx}^{-1})^{-1} \in Z^{-1}\). Let \(v = \overline{tx}\). Note that
		\(\overline{vx^{-1}} = \overline{\overline{tx}x^{-1}} = t\). Then \(g =
		vx^{-1} \overline{vx^{-1}}^{-1} \in S\), and so \(Z^{-1} \subseteq S\).

		Let \(g = tx^{-1} \overline{tx^{-1}}^{-1} \in S\). Then \(g^{-1} =
		\overline{tx^{-1}} xt^{-1}\). Let \(v = \overline{t x^{-1}}\). Then
		\(\overline{vx} = t\), and so \(g^{-1} = vx \overline{vx}^{-1} \in Z\). We
		can conclude that \(S \subseteq Z^{-1}\).

		The fact that \(Z\) is finite follows from the fact that \(T\) and
		\(\Sigma\) are finite. Let \(t_0\) be the unique element of \(T \cap H\).
		Let \(h \in H\) (this will usually be \(1\), but does not need to be). Then
		\(t_0^{-1} h t_0 = a_1 \cdots a_n\), for some \(a_1, \ \ldots, \ a_n \in
		\Sigma^\pm\). Let \(t_i = \overline{a_1 \cdots a_i}\) for each \(i \in \{1, \
		\ldots, \ n\}\), and note \(t_n = t_0\). We have
		\[
			h = (t_0 a_1 t_1^{-1}) (t_1 a_2 t_2^{-1}) \cdots (t_{n - 1} a_n t_n^{-1}).
		\]
		Note that \(\overline{t_i a_{i + 1}} = \overline{\overline{a_1 \cdots
		a_i}a_{i + 1}} = t_{i + 1}\), and so
		\[
			h = (t_0 a_1 \overline{t_0 a_1}^{-1}) (t_1 a_2 \overline{t_1 a_2}^{-1})
			\cdots (t_{n - 1} a_n \overline{t_{n - 1}a_n}^{-1}).
		\]
		Each of the parenthesised terms lie in \(Z\) if \(a_i \in \Sigma\), or
		\(S\) if \(a_i \in \Sigma^{-1}\). Since \(S = Z^{-1}\), we have \(h \in
		\langle Z \rangle\).
	\end{proof}

	The proof of Lemma \ref{schreier_lem} induced a normal form for the finite
	index subgroup, with respect to the Schreier generating set. We now give a
	formal definition of this normal form.

	\begin{dfn}
		\label{Schreier_norm_form_def}
		Let \(G\) be a group, generated by a finite set \(\Sigma\), \(H\) be a finite
		index subgroup of \(G\), and \(T\) be a right transversal of \(H\) in \(G\).
		Let \(Z\) be the Schreier generating set for \(H\). Fix a normal form
		\(\eta\) for \((G, \ \Sigma)\).

	  We define the \textit{Schreier normal form} \(\zeta\) for \((H, \ Z)\),
	  with respect to \(\eta\), as follows. Let \(h \in H\), and suppose \(h \eta = a_1
	  \cdots a_n\), where \(a_1, \ \ldots, \ a_n \in \Sigma^\pm\). Let \(t_0\) be
	  the unique element of \(T \cap H\), and define \(t_i = \overline{a_1 \cdots
	  a_i}\). Define \(h \zeta\) by
		\begin{equation}
			\label{Schreier_nrom_form_eqn}
			h \zeta = (t_0 a_1 \overline{t_0 a_1}^{-1}) (t_1 a_2 \overline{t_1
			a_2}^{-1}) \cdots (t_{n - 1} a_n \overline{t_{n - 1}a_n}^{-1}).
		\end{equation}
		The fact that this indeed defines an element of \(H\), and equals \(h\) is
		contained in the proof of Lemma \ref{schreier_lem}.
	\end{dfn}

	If the normal form from the finite index overgroup is regular or quasigeodesic,
	then the Schreier normal form is regular or quasigeodesic, respectively. The
	latter requires an additional lemma that we prove later, however we can show
	that regularity is preserved without additional results.

	\begin{lem}
		\label{Schreier_norm_form_reg_lem}
		Let \(G\) be a group, generated by a finite set \(\Sigma\), \(H\) be a finite
		index subgroup of \(G\), and \(T\) be a right transversal of \(H\) in \(G\).
		Let \(Z\) be the Schreier generating set for \(H\). Fix a normal form
		\(\eta\) for \((G, \ \Sigma)\).

		Let \(\zeta\) be the Schreier normal form with respect to \(\eta\), as in
		\eqref{Schreier_nrom_form_eqn}. If
		\(\eta\) is regular with respect to \(\Sigma\), then \(\zeta\) is regular
		with respect to \(Z\).
	\end{lem}

	\begin{proof}
		We will extend \(\zeta\) to the whole of \(G\), with respect to the
		generating set \(Z \cup \{tx u^{-1} \mid u, \ t \in T, \ x \in \Sigma\}\).
		Let \(g \in G\), and suppose \(t_0 g t_0^{-1} \eta = a_1 \cdots a_n\) where
		each \(a_i \in \Sigma^\pm\). Define \(\tilde{\zeta} \colon G \to ((Z \cup
		\{tx u^{-1} \mid u, \ t \in T, \ x \in \Sigma\})^\pm)^\ast\) by
		\[
			g \tilde{\zeta} =  (t_0 a_1 \overline{t_0 a_1}^{-1}) (t_1 a_2
			\overline{t_1 a_2}^{-1}) \cdots (t_{n - 1} a_n t_0^{-1}).
		\]
		Note that \(\tilde{\zeta}\) is an extension of \(\zeta\). We will first
		show that \(\tilde{\zeta}\) is regular, then use an intersection to show
		\(\zeta\) is regular.

		Consider a finite state automaton \(\mathcal A\) that accepts \(\im \eta\),
		with set of states \(Q\), start state \(q_0\), and set \(F\) of accept
		states. We will construct a new finite state automaton \(\mathcal B\) to
		accept \(\im \tilde{\zeta}\). Our set of states will be \((Q \times T \times
		\{0, \ 1\}) \cup \{\lambda\}\), where \(\lambda\) is a new state, our start
		state will be \((q_0, \ t_0, \ 0)\), and \(\lambda\) will be our only
		accept state. For each transition \((p, \ a) \to q\) in \(\mathcal A\), and
		each \(t \in T\), define the following transitions in \(\mathcal B\):
		\begin{align*}
			& ((p, \ t, \ 0), \ a) \to (q, \ \overline{ta}, \ 1), \\
			& ((q, \ \overline{ta}, \ 1), \ \overline{ta}^{-1}) \to (q, \ \overline{ta}, \ 0).
		\end{align*}
		For each \(q \in Q\) and \(t \in T\), we also have a transition
		\[
			((q, \ t, \ 1), \ t_0^{-1}) \to \lambda.
		\]
		By construction, whenever we read \(ta\), we must follow with
		\(\overline{ta}^{-1}\), unless we are going to the accept state (at the end
		of the word), in which case we follow with \(t_0^{-1}\). As a result,
		\(\mathcal B\) only accepts words in \(\im \tilde{\zeta}\). Conversely,
		\(\mathcal B\) accepts any word in \(\im \eta\) after its conversion into a
		word in \(\im \tilde{\zeta}\), and we can therefore conclude that \(\mathcal
		B\) accepts \(\im \tilde{\zeta}\).

		We have that \(\im \zeta = \im \tilde{\zeta} \cap (Z^\pm)^\ast\). As an
		intersection of regular languages, this is regular.
	\end{proof}

\section{EDT0L languages about a distinguished letter}
	\label{distinguished_letter_sec}
	Recall that we denote a solution \((g_1, \ \ldots, \ g_n)\) to a system of
	equations in a group \(G\) using the word \((g_1 \eta) \# \cdots \# (g_n
	\eta)\). In order to show that groups where systems of equations have EDT0L
	solution languages are
	closed under certain types of extension (such as direct products), we are
	required to prove Lemma \ref{intermesh_EDT0L_lem}, which allows us to
	concatenate in parallel two EDT0L languages where every word is of the form
	\(u_0 \# \cdots \# u_n\).

	The following lemma allows us to use different symbols for each \(\#\) that
	delimits the group elements, rather than the same one each time. The proof is
	joint work with Alex Evetts. When used in conjunction with Lemma
	\ref{EDT0L_diff_alphabet_lem} and the fact that the class of EDT0L languages
	is closed under images under homomorphisms, we can use this to show that the
	solution language remains EDT0L when restricted to a subset of variables,
	rather than all of them.

	\begin{lem}
		\label{EDT0L_diff_hash_lem}
		Let \(n \in \mathbb{Z}_{> 0}\), \(\{\#, \ \#_1, \ \ldots, \
		\#_n\}\) be a set of formal symbols, and \(\Delta\) be an alphabet, such
		that \(\#, \ \#_1, \ \ldots, \ \#_n \notin \Delta\). Let \(A\) be a set of
		\(n\)-tuples of words over \(\Delta\). Define languages \(L\) and \(M\) over
		\(\Delta \cup \{\#\}\) and \(\Delta \cup \{\#_1, \ \ldots, \#_n\}\),
		respectively, by
		\begin{align*}
			L & = \{w_1 \# w_2 \# \cdots \# w_n \ | \ (w_1, \
			\ldots, \ w_n) \in A\} \\
			M & = \{w_1 \#_1 w_2 \#_2 \cdots
			\#_{n - 1} w_n \#_n \ | \   (w_1, \ \ldots, \ w_n) \in A\}.
		\end{align*}
		Let \(f \colon \mathbb{Z}_{\geq 0} \to \mathbb{Z}_{\geq 0}\). Then
		\begin{enumerate}
			\item The language \(L\) is EDT0L if and only if \(M\) is;
			\item There exists an EDT0L system for \(L\) that is constructible in
			\(\ns(f)\) if and only if such an EDT0L system for \(M\) exists.
		\end{enumerate}
	\end{lem}

	\begin{proof}
		Applying the monoid homomorphism \( \#_1, \ \ldots, \ \#_{n - 1} \mapsto \#,
		\ \#_n \mapsto \varepsilon\) maps \(M\) to \(L\), so the backward directions
		of (1) and (2) follow by Lemma \ref{EDT0L_closure_properties_lem}.

		Suppose \(L\) is EDT0L. We will first show that
		\[
			N := \{w_1 \#_1 w_2 \#_2 \cdots \#_{n - 1} w_n
			\ | \ (w_1, \ \ldots, \ w_n) \in A\}
		\]
		is EDT0L. Consider an EDT0L system \(\mathcal H_L = (\Sigma \sqcup \{\#\}, \
		C, \ \perp, \ \mathcal{R})\) that accepts \(L\), and that is constructible
		in \(\ns(f)\). Note that we can assume our start word is a single letter,
		instead of a word \(\omega\) by adding an additional letter \(\perp\), and
		preconcatenating the rational control with an endomorphism \(\perp \mapsto
		\omega\). Let \(B \subseteq \End(C^\ast)\) be the (finite) set over which
		\(\mathcal R\) is a regular language.

		We will construct a new EDT0L system from \(\mathcal H_L\) which will accept
		\(M\). Let \(C_\text{ind} = \{c^{i, i + 1, \ldots, j} \mid c \in C, \ i, \ j
		\in \{1, \ \ldots, \ n\}\}\) be the set of symbols obtained by indexing
		elements of \(C\) with a section of the sequence \((1, \ \ldots, \ n)\),
		including the empty sequence (if \(i > j\)). By convention, we will consider
		a letter \(c \in C\) indexed by the empty sequence to be equal to \(c\), and
		so \(C \subseteq C_\text{ind}\). Our extended alphabet will be
		\(C_\text{ind}\). Let \(\phi \in B\). Define \(\Phi_\phi \subseteq
		\End(C_\text{ind}^\ast)\) to be the set of all endomorphisms \(\psi\)
		defined by
		\[
			c^{i, \ldots, j} \psi = x_1^{i_{11}, \ldots, i_{1k_1}}
			x_2^{i_{21}, \ldots i_{2k_2}} \cdots x_r^{i_{r1}, \ldots, i_{rk_r}},
		\]
		where \(x_1 \cdots x_r = c \phi\), and \((i_{11}, \ \ldots, \ i_{rk_r}) =
		(i, \ \ldots, \ j)\). Note that some (or all) of the sequences may be empty.
		Let \(\bar{\mathcal R}\) be the regular language of endomorphisms of
		\(C_\text{ind}^\ast\) obtained from \(\mathcal R\) by replacing each \(\phi
		\in B\) with \(\Phi_\phi\). The EDT0L system \(\mathcal H_M = (\Sigma \cup
		\{\#_1, \ \ldots, \ \#_n\}, \ C_\text{ind}, \ \perp_{1, \ldots, n},
		\bar{\mathcal R})\) will only accept words of the form \(a_1^{i_{11},
		\ldots, i_{1k_1}} \cdots a_r^{i_{r1}, \ldots, i_{rk_r}}\), where \((i_{11},
		\ \ldots, \ i_{rk_r}) = (1, \ \ldots, \ n)\), and \(a_1 \cdots a_r \in L\).
		However, since our alphabet is \(\Sigma \cup \{\#_1, \ \ldots, \ \#_n\}\),
		it can only accept words over that alphabet, which are precisely words of
		the form \(w_0 \#_1 \cdots \#_n w_n\), where \(w_1 \# \cdots \# w_n \in L\),
		and thus will accept \(M\).

		It now remains to show \(\mathcal H_M\) is constructible in \(\ns(f)\). It
		doesn't require extra memory beyond a constant to add \(\perp\) as the start
		symbol. To write down the new extended alphabet \(C_\text{ind}\), we just
		proceed as we would when constructing \(\mathcal H_L\), but whenever we
		write a symbol \(c\), we also write all of the indexed versions. To do this
		we just need to record the letter \(c\) we are on, along with the previous
		index written, so this is still possible in \(\ns(f)\).

		To output \(\bar{\mathcal R}\), we simply proceed with writing down the
		finite state automaton that accepts \(\mathcal R\), and replace each edge
		labelled by \(\phi \in B\) with a set of edges between the same states,
		labelled with each \(\psi \in \Phi_\phi\). To do this, we can compute
		\(\Phi_\phi\), store it, and remove each \(\psi \in \Phi_\phi\) from the
		memory as we write it. This will require \(n\) times as much memory as
		writing down \(\mathcal R\), but since \(n\) is a constant, it is
		constructible in \(\ns(f)\).
	\end{proof}

	We introduce the concept of a \((\#_1, \ \ldots, \
	\#_n)\)-separated EDT0L system, which is key in the proof of Lemma
	\ref{intermesh_EDT0L_lem}.

	\begin{dfn}
		Let \(\Sigma\) be an alphabet, and \(\#_1, \ \ldots, \ \#_n \in \Sigma\). A
		\textit{\((\#_1, \ \ldots, \ \#_n)\)-separated EDT0L system} is an EDT0L
		system \(\mathcal{H}\), with a start word of the form \(\omega_0 \#_1
		\omega_1 \#_2 \cdots \#_n \omega_n\), where \(\omega_i \in (\Sigma
		\backslash \{\#_1, \ \ldots, \ \#_n\})^\ast\) for all \(i\), and such that
		\(\#_i \phi^{-1} = \{\#_i\}\), for all \(i\), and every \(\phi\) in the
		rational control.
	\end{dfn}

  Most of the proof of Lemma \ref{intermesh_EDT0L_lem} involves showing that
  EDT0L languages where every word contains precisely \(n\) occurences of the
  letter \(\#\) are accepted by \((\#, \ \ldots, \ \#)\)-separated EDT0L
  systems. To start this proof, we need the following lemma. The proof is very
  similar to the proof of Lemma \ref{EDT0L_diff_hash_lem}.

	\begin{lem}
    \label{EDT0L_diff_alphabet_lem}
    Let \(L\) be an EDT0L language over an alphabet \(\Sigma\), such that every
    word in \(L\) contains precisely \(n\) occurrences of the letter \(\#\),
    where \(n \in \mathbb{Z}_{\geq 0}\). Let \(f \colon \mathbb{Z}_{\geq 0} \to
    \mathbb{Z}_{\geq 0}\). Let \(\bar{\Sigma}^0 = \{\bar{a}^0 \mid a \in
    \Sigma\}, \ \ldots, \ \bar{\Sigma}^n = \{\bar{a}^n \mid a \in \Sigma\}\) be
    pairwise disjoint copies of \(\Sigma\), all disjoint from \(\Sigma\). Let
    \(\varphi^i \colon \Sigma^\ast \to (\bar{\Sigma}^i)^\ast\) be the free monoid
    homomorphism defined by \(a \varphi^i = \bar{a}^i\) for all \(a \in \Sigma\).
    Then
		\begin{enumerate}
			\item The language
        \[
          M = \{(w_0 \varphi^0) \# \cdots \# (w_n \varphi^n) \mid w_0 \# \cdots
          \# w_n \in L\}
        \]
        is EDT0L;
      \item If there is an EDT0L system for \(L\) that is constructible in
        \(\ns(f)\), then there is an EDT0L system for \(M\) that is
        constructible in \(\ns(f)\), which has a single-letter start word.
		\end{enumerate}
	\end{lem}

  \begin{proof}
		We will assume without loss of generality that \(\# \cdots \# \notin L\)
		(that is, the length of every word in \(L\) is at least \(n + 1\)); if this
		is not the case, the fact that finite languages are EDT0L and the fact that
		EDT0L languages are closed under finite unions will give the result (Lemma
		\ref{EDT0L_closure_properties_lem}). Consider an EDT0L system \(\mathcal
		H_L = (\Sigma \sqcup \{\#\}, \ C, \ \perp, \ \mathcal{R})\) that accepts
		\(L\), and that is constructible in \(\ns(f)\). As in the proof of Lemma
		\ref{EDT0L_diff_hash_lem}, we can assume our start word is a single letter,
		instead of a word \(\omega\) by adding an additional letter \(\perp\), and
		preconcatenating the rational control with an endomorphism \(\perp \mapsto
		\omega\). Let \(B \subseteq \End(C^\ast)\) be the (finite) set over which
		\(\mathcal R\) is a regular language.

    We will construct a new EDT0L system from \(\mathcal H_L\) which will accept
    the language \(K\) of all words obtained from a word in \(L\) by replacing
    each letter \(a \in \Sigma \cup \{\#\}\) with \(\bar{a}^i\) for some \(i\),
    such that concatenating the indices of all letters in a word in \(K\) gives
    a non-decreasing sequence containing that numbers \(1, \ \ldots, \ n\) (that
    is, of the form \((0, \ \ldots, \ 0, \ 1, \ \ldots, \ 1, \ \ldots, \ n, \
    \ldots, \ n)\)). Note that each letter in a word in \(K\) may have a
    different index \(i\). Let \(C_\text{ind} = \{c^{i, i + 1, \ldots, j} \mid c
    \in C, \ i, \ j \in \{0, \ \ldots, \ n\}\}\) be the set of symbols obtained
    by indexing elements of \(C\) with a non-empty section of the sequence \((0,
    \ \ldots, \ n)\). Note that this differs from the proof of Lemma
    \ref{EDT0L_diff_hash_lem} where the empty sequence was permitted.
		Our extended alphabet will be
    \(C_\text{ind}\). Let \(\phi \in B\). Define \(\Phi_\phi \subseteq
    \End(C_\text{ind}^\ast)\) to be the set of all endomorphisms \(\psi\)
    defined by
		\[
			c^{i, \ldots, j} \psi = x_1^{i_{11}, \ldots, i_{1k_1}}
			x_2^{i_{21}, \ldots i_{2k_2}} \cdots x_r^{i_{r1}, \ldots, i_{rk_r}},
		\]
    where \(x_1 \cdots x_r = c \phi\), and \((i_{11}, \ \ldots, \ i_{rk_r})\) is
    a non-decreasing sequence such that \(\{i_{11}, \ \ldots, \ i_{rk_r} \} =
    \{i, \ \ldots, \ j\}\), and such that each subsequence \((i_{p1}, \ \ldots,
    \ i_{pk_p})\) for some \(p\) is strictly increasing. That is, we split up
    the sequence \((i, \ \ldots, \ j)\) across the word \(x_1 \cdots x_r\),
    potentially adding some repeats of integers across two letters, but never
    within the index of one letter. Note that some (or all) of the sequences may
    be empty. Let \(\bar{\mathcal R}\) be the regular language of endomorphisms
    of \(C_\text{ind}^\ast\) obtained from \(\mathcal R\) by replacing each
    \(\phi \in B\) with \(\Phi_\phi\), and let \(\bar{B}\) analogous set
    obtained from \(B\).

		Let \(\mathcal H_K = (\bar{\Sigma}^0 \cup \cdots \cup \bar{\Sigma}^n \cup
		\{\#^0, \ \ldots, \ \#^n\}, \ C_\text{ind}, \ \perp_{0, \ldots, n}, \
		\bar{\mathcal R})\). Unlike the system created in the proof of Lemma
		\ref{EDT0L_diff_hash_lem}, if \(\psi \in \bar{B}^\ast\) is such that there
		exists \(\phi \in \bar{B}^\ast\) such that \(\psi \phi \in \bar{\mathcal
		R}\), then it is now possible for two letters in a word \(\perp \psi\) to
		have the same index (as we can have such a word of the form \(c^0 c^1 d^1
		c^1 d^2\) if \(n = 2\)). This is not a problem as the only way two of the
		same letters can have the same index is if they have a single-number index
		(that is, there is a subword of the form \(c^i \nu c^i\), for some \(\nu \in
		C_\text{ind}^\ast\)). In such a case, we will want to map both \(c^i\)s to
		the same place (as all of their images should have only \(i\) as their
		index). In addition, we will not miss any words in \(L\) due to them being
		too short to fit all indices on; we start with \(n + 1\) indices, and all
		words in \(L\) have at least length \(n + 1\), by our assumption at the
		beginning of the proof.

		The EDT0L system \(\mathcal H_K\) will only accept words of the form
		\(a_1^{i_1} \cdots a_r^{i_r}\), where \((i_1, \ \ldots, \ i_r)\) is
		non-decreasing, \(\{i_{1}, \ \ldots, \ i_{r}\} \subseteq \{0, \ \ldots, \
		n\}\) (we can have letters mapped to \(\varepsilon\), so we may lose some
		indices), and \(a_1 \cdots a_r \in L\). However, since our alphabet is
		\(\bar{\Sigma}^0 \cup \cdots \cup \bar{\Sigma}^n \cup \{\#^0, \ \ldots, \
		\#^n\}\), it can only accept words in \(K\), and so \(K\) is EDT0L, accepted
		by the system \(\mathcal H_K\).

    We will now show that \(\mathcal H_K\) is constructible in \(\ns(f)\). As
    in the proof of Lemma \ref{EDT0L_diff_hash_lem}, we don't require extra
    memory beyond a constant to add \(\perp\) as the start symbol.  To output
    the new extended alphabet \(C_\text{ind}\), we just proceed as we would
    when constructing \(\mathcal H_L\), but whenever we write a symbol \(c\),
    we also write all of the indexed versions. To do this we just need to
    record the letter \(c\) we are on, along with the previous index written,
    so this is still possible in \(\ns(f)\).

		To output \(\bar{\mathcal R}\), we simply proceed with writing down the
		finite state automaton that accepts \(\mathcal R\), and replace each edge
		labelled by \(\phi \in B\) with a set of edges between the same states,
		labelled with each \(\psi \in \Phi_\phi\). To do this, we can compute
		\(\Phi_\phi\), store it, and remove each \(\psi \in \Phi_\phi\) from the
		memory as we write it. This will require \(n\) times as much memory as
		writing down \(\mathcal R\), but since \(n\) is a constant, it is
		constructible in \(\ns(f)\).

    Note that \(M\) obtained by intersecting \(K\) with the regular language
    \((\bar{\Sigma}^0)^\ast \#^1 \cdots \#^n (\bar{\Sigma}^n)^\ast\), and then
    applying the free monoid homomorphism defined by \(\#^i \mapsto \#\) for all
    \(i\). Thus by Lemma \ref{EDT0L_closure_properties_lem}, \(M\) is accepted
    by an EDT0L system that is constructible in \(\ns(f)\).
  \end{proof}

	Using Lemma \ref{EDT0L_diff_alphabet_lem}, we can now show that EDT0L
	languages where every word contains \(n\) occurences of the letter \(\#\)
	are always accepted by \((\#, \ \ldots, \ \#)\)-separated EDT0L systems.

	\begin{lem}
		\label{n_hash_sep_lem}
		Let \(L\) be an EDT0L language, such that every word in \(L\) contains
		precisely \(n\) occurrences of the letter \(\#\), where \(n \in
		\mathbb{Z}_{\geq 0}\). Let \(f \colon \mathbb{Z}_{\geq 0} \to
		\mathbb{Z}_{\geq 0}\). Then
		\begin{enumerate}
			\item There is a \((\#, \ \ldots, \ \#)\)-separated EDT0L system
			\(\mathcal H\) that accepts \(L\).
			\item If an EDT0L system for \(L\) is constructible in \(\ns(f)\), then
			\(\mathcal H\) is constructible in \(\ns(f)\).
		\end{enumerate}
	\end{lem}


	\begin{proof}
    Let \(\bar{\Sigma}^0, \ \ldots, \ \bar{\Sigma}^n\) be the disjoint copies
    of \(\Sigma\)  and let \(\varphi^0, \ \ldots, \ \varphi^n\) be the
    corresponding maps, as defined in Lemma \ref{EDT0L_diff_alphabet_lem}. Let
    \(\Delta = (\bar{\Sigma}^0)^\ast \cup \cdots \cup (\bar{\Sigma}^n)^\ast
    \cup \{\#\}\). Suppose an EDT0L system for \(L\) is constructible in
    \(\ns(f)\). By Lemma \ref{EDT0L_diff_alphabet_lem},
    \[
      M = \{(w_0 \varphi^0) \# \cdots \# (w_n \varphi^n) \mid w_0 \# \cdots
      \# w_n \in L\}
    \]
    is EDT0L, and accepted by an EDT0L system \(\mathcal H_M = (\Delta, \ C, \
    \perp, \ \mathcal R)\) that is constructible in \(\ns(f)\). Note that we
    can assume that \(\#\) is fixed by endomorphisms in \(\mathcal R\) by
    adding an additional symbol \(\#'\) to the extended alphabet to replace
    \(\#\), and then post-composing the rational control with a map that sends
    \(\#'\) to \(\#\) and fixes everything else. This will not affect the space
    complexity in which \(\mathcal H_M\) can be constructed.

    Let \(C_0, \
    \ldots, \ C_n\) be pairwise disjoint copies of \(C\), and let \(\Delta_0, \
    \ldots, \ \Delta_n\) be the copies of \(\Delta\) sitting inside \(C_0, \
    \ldots, \ C_i\), respectively. Let \(\theta_i \colon \Delta^\ast \to
    (\Delta_i)^\ast\) be the monoid homomorphism defined by mapping a letter in
    \(\Delta\) to its copy in \(\Delta_i\). We will
    show that
    \[
      K = (M \theta_0) \# (M \theta_1) \# \cdots \# (M \theta_n)
    \]
    is EDT0L over the alphabet \(\Delta_0 \cup \cdots \cup \Delta_n \cup
    \{\#\}\), and accepted by a \((\#, \ \ldots, \ \#)\)-separated EDT0L system
    that is constructible in \(\ns(f)\). Most of this can be done by Lemma
    \ref{EDT0L_closure_properties_lem}, however to show the stronger property
    that that a \((\#, \ \ldots, \ \#)\)-separated EDT0L system exists, we must
    do this directly. The proof is largely the same. For each \(i\), let
    \(\mathcal H_{M \theta_i} = (\Delta_i, \ C_i, \ \perp_i, \ \mathcal{R}_i)\)
    be an EDT0L system for \(M \theta_i\) that is constructible in \(\ns(f)\)
    (such a system exists by replacing each \(c \in C\) with the image in
    \(C_i\), and then updating the maps to act on \(C_i^\ast\) instead of
    \(C^\ast\)).


		Our extended alphabet will be \(D = C_0 \cup \cdots \cup C_n \cup \{\#\}\).
		We now define the rational control for our EDT0L system for \(K\). Let \(B
		\subseteq \mathcal{R}\) be a finite set over which \(\mathcal{R}\) can be
		expressed as a regular language. For each \(\phi \in B\), define
		\(\hat{\phi}_i \in \End(D^\ast)\) by
    \begin{align*}
      a \hat{\phi} & = \left\{
      \begin{array}{cl}
       b \phi \theta_i & a \in C_i, \ a = b \theta_i \\
       a & a = \#.
      \end{array}
      \right.
    \end{align*}
    Let \(\mathcal{S}\) be the regular language of endomorphisms of \(D^\ast\)
    obtained by replacing each occurence of a label \(\phi\) within a finite state
    automaton for \(\mathcal{R}\) with \(\hat{\phi}\). Our start word will be
    \(\perp_0 \# \perp_1 \# \cdots \# \perp_n\). By construction,
    \(\mathcal S\) manipulates each \(\perp_i\) in parallel exactly the same
    way that \(\mathcal{R}\) affects \(\perp\), except remaining in the correct
    copy \(C_i\).

    Let \(\mathcal{H}_K = (\Delta_0 \cup \cdots \cup \Delta_n \cup \{\#\}, \ D,
    \ \perp_0 \# \perp_1 \# \cdots \# \perp_n, \ \mathcal S)\). By
    construction, \(\mathcal{H}_K\) accepts \(K\) and is \((\#, \ \ldots, \
    \#)\)-separated.  Additionally, as each of \(\mathcal{R}_1, \ \ldots, \
    \mathcal{R}_n\) is constructible in \(\ns(f)\), so is the rational control
    \(\mathcal{S}\) of \(\mathcal H_K\).  The extended alphabet \(D\) can be
    output using the same information required to construct each \(C_i\), which
    is just the same information needed to output \(C\), and so \(D\) is
    constructible in \(\ns(f)\).

    Define the monoid homomorphism \(\Psi\) by its action on the letters in
    \(\Delta_0 \cup \cdots \cup \Delta_n \cup \{\#\}\). Note that
    \(\theta_i^{-1}\) and \(\varphi_i^{-1}\) are well-defined, because
    \(\theta_i\) and \(\varphi_i\) are bijections on the letters.
    \begin{align*}
      \Psi \colon (\Delta_0 \cup \cdots \cup \Delta_n \cup
      \{\#\})^\ast & \to \Sigma^\ast \\
      a & \mapsto \left\{
      \begin{array}{cl}
        a (\theta_i)^{-1} (\varphi^i)^{-1} & a \in M \theta_i \text{ and }
        a (\theta_i)^{-1} \in \bar{\Sigma}^i \text{ for some } i \\
        a & a = \# \\
        \varepsilon & \text{otherwise}.
      \end{array}
      \right.
    \end{align*}
    By construction, \(K \Psi = L\). Thus by post-composing the rational control
    of \(\mathcal H_K\) with \(\Psi\), we obtain a \((\#, \ \ldots, \
    \#)\)-separated EDT0L system for \(L\). Moreover, this addition will not
    affect the space complexity of \(\mathcal H_K\), which is \(\ns(f)\).
	\end{proof}

		We are now able to prove the main result of this section.

		\begin{lem}
			\label{intermesh_EDT0L_lem}
			Let \(L\) and \(M\) be EDT0L languages, such that every word in \(L \cup M\)
			contains precisely \(n\) occurrences of the letter \(\#\). Let \(f \colon
			\mathbb{Z}_{\geq 0} \to \mathbb{Z}_{\geq 0}\). Then
			\begin{enumerate}
				\item The language
				\[
					N = \{u_0 v_0 \# \cdots \# u_n v_n \ | \ u_0 \# \cdots
					\# u_n \in L, \ v_0 \# \cdots \# v_n \in M\},
				\]
				is EDT0L;
				\item If EDT0L systems for \(L\) and \(M\) are constructible in
				\(\ns(f)\), then an EDT0L system for \(N\) is constructible in \(\ns(f)\).
			\end{enumerate}
		\end{lem}

		\begin{proof}
	 		By Lemma \ref{n_hash_sep_lem} and Lemma \ref{EDT0L_diff_hash_lem}, we
	 		have that \(L\) and \(M\) are accepted by EDT0L systems \(\mathcal{H}_L\)
	 		and \(\mathcal{H}_M\), with start words \(\omega_0 \#_1 \cdots \#_n
	 		\omega_n\) and \(\nu_0 \#_1 \cdots \#_n \nu_n\), respectively, such that
	 		nothing other than \(\#_i\) is mapped to \(\#_i\) within both
	 		\(\mathcal{H}_L\) and \(\mathcal{H}_M\). Suppose also that these systems
	 		are constructible in \(\ns(f)\). Let \(C_L\) and \(C_M\) be the extended
	 		alphabets of \(\mathcal{H}_L\) and \(\mathcal{H}_M\), and let
	 		\(\Sigma_L\) and \(\Sigma_M\) be the terminal alphabets. Without loss of
	 		generality assume \(C_L \backslash \Sigma_L\) and \(C_M \backslash
	 		\Sigma_M\) are disjoint. Let \(\mathcal{R}_L\) and \(\mathcal{R}_M\) be
	 		the rational controls, and let \(B_L\) and \(B_M\) be the finite sets of
	 		endomorphisms over which \(\mathcal R_L\) and \(\mathcal R_M\) are
	 		regular languages.

			Let \(\Sigma = \Sigma_L \cup \Sigma_M\), and let \(C = C_L \cup C_M\). For
			each \(\phi \in B_L\), define \(\bar{\phi} \in \End(C^\ast)\) by
			\[
				c \bar{\phi} = \left\{
				\begin{array}{cl}
					c \phi & c \in C_L \\
					c & c \notin C_L.
				\end{array}
				\right.
			\]
			Define \(\bar{\phi}\) for each \(\phi\) in \(B_M\) analogously, and extend
			the bar notation to composition of functions, that is, \(\overline{\phi
			\psi} = \bar{\phi} \bar{\psi}\). Let \(\mathcal{R} = \{\bar{\phi} \mid \phi
			\in \mathcal{R}_L \cup \mathcal{R}_M\}\), and note that \(\mathcal{R}\) is a
			regular language over some alphabet of endomorphisms. Thus, \(N\) is
			accepted by the EDT0L system \((\Sigma, \ C, \ \omega_0 \nu_0 \#_1 \cdots
			\#_n \omega_n \nu_n, \ \mathcal{R})\), as required.

			Suppose there exist EDT0L systems for \(L\) and \(M\), which are
			constructible in \(\ns(f)\). By Lemma \ref{n_hash_sep_lem} and Lemma
			\ref{EDT0L_diff_hash_lem}, \(\mathcal H_L\) and \(\mathcal H_M\) are also
			constructible in \(\ns(f)\). We can construct \(C\) with the memory
			required to construct \(C_M\) and \(C_L\). The set \(\{\bar{\phi} \mid
			\phi \in \mathcal B_L\}\) is constructible in \(\ns(f)\), by following the
			construction of \(\mathcal R_L\), but writing a \(\bar{\phi}\) instead of
			a \(\phi\), for each occurence of \(\phi \in B_L\). By symmetry,
			\(\{\bar{\phi} \mid \phi \in \mathcal B_M\}\) is constructible in
			\(\ns(f)\). Since \(\mathcal R\) is the union of these sets, we can
			construct \(\mathcal R\) in \(\ns(f)\) by Lemma
			\ref{reg_lang_closure_prop_lem}.
		\end{proof}

\section{Equations in extensions}
	\label{eqns_in_exts_sec}

	This section shows that the class of groups where systems of equations have
	EDT0L solution languages is closed under various extensions, including wreath
	products with finite groups and direct products. These facts are used in the
	proof of Theorem \ref{VDP_hyp_thm} on groups that are virtually a direct
	product of hyperbolic groups.

	Furthermore, in Proposition \ref{virt_eqn_prop} we deal with systems of
	equations with rational constraints in finite extensions, if twisted equations
	in a finite index normal subgroup have EDT0L solutions. This is used to show
	that systems of equations with rational constraints in virtually abelian
	groups are EDT0L.

	The proof of the following is based on the proof of Lemma 3.9 in \cite{VAEP}.

	\begin{lem}
		\label{lem:virtual_sols}
		Let \(G\) be a group, and \(T\) be a finite transversal of a normal subgroup
		\(H\) of finite index. Let \(\Omega\) be the group of automorphisms of \(H\)
		induced by conjugating \(H\) by elements of \(G\). Let \(\mathcal{S}\) be
		the solution set to a finite system \(\mathcal E_G\) of equations with
		rational constraints in \(n\) variables in \(G\). Then there is a finite set
		\(B \subseteq T^n\), and for each \(\mathbf{t} = (t_1, \ \ldots, \ t_n) \in
		B\), there is a solution set \(A_\mathbf{t}\) to a system \(\mathcal E_{H,
		\mathbf{t}}\) of \(\Omega\)-twisted equations with rational constraints in
		\(H\), such that
		\[
			\mathcal{S} = \bigcup_{(t_1, \ \ldots, \ t_n) \in B}
			\left\{(h_1 t_1, \ \ldots, \  h_n t_n) \mid
			(h_1, \ \ldots, \ h_n) \in A_{(t_1, \ \ldots, \ t_n)} \right\}.
		\]
	\end{lem}

	\begin{proof}
		Let
		\begin{align}
			\label{virt_eqn}
			X_{i_{1j}}^{\epsilon_{1j}} g_{1j} \cdots X_{i_{pj}}^{\epsilon_{pj}} g_{pj}
			= 1
		\end{align}
		be a system \(\mathcal E_G\) of equations in \(G\), with a set \(\{R_{X_1},
		\ \ldots, \ R_{X_n}\}\) of rational constraints, where \(X_1, \ \ldots, \
		X_n\) are the variables, and \(j \in \{1, \ \ldots, \ k\}\). Let \(\mathcal
		S\) be the solution set. Note that we can assume that these equations start
		with variables by conjugating leading constants to the right. For each
		\(X_i\), define new variables \(Y_i\) over \(H\), and \(Z_i\) over \(T\),
		such that \(X_i = Y_i Z_i\). For each constant \(g_i\), we have \(g_i = h_i
		t_i\), for some \(h_i \in H\) and \(t_i \in T\), and so substituting these
		into \eqref{virt_eqn} gives that \(\mathcal E_G\) is equivalent to
		\begin{align}
			\label{virt_subbed_in_eqn}
			(Y_{i_{1j}} Z_{i_{1j}})^{\epsilon_{i_{1j}}} h_{1j} t_{1j} \cdots
			(Y_{i_{pj}} Z_{i_{pj}})^{\epsilon_{i_{pj}}} h_{pj} t_{pj} = 1.
		\end{align}
		For all \(g \in G\), define \(\psi_g \colon G \to G\) by \(h \psi_g = gh
		g^{-1}\). Note that \(\psi_g \restriction_H \in \Omega\) for all \(g \in
		G\), by definition. By abusing notation, we can define \(\psi_{Z_i}\) for
		each \(i\). For all \(i \in \{1, \ \ldots, \ n\}\), and \(j \in \{1, \
		\ldots, \ k\}\) define
		\[
			\delta_{ij} = \left\{
				\begin{array}{cl}
					0 & \epsilon_{ij} = 1 \\
					1 & \epsilon_{ij} = -1.
				\end{array}
			\right.
		\]
		We can use this notation to rearrange \eqref{virt_subbed_in_eqn} into
		\begin{align}
			\label{virt_eqn_2}
			(Y_{i_{1j}}^{\epsilon_{1j}} \psi_{Z_{i_{1j}}}^{\delta_{1j}}) Z_{i_{1j}}^{\epsilon_{1j}}
			h_{1j} t_{1j} \cdots
			(Y_{i_{pj}}^{\epsilon_{pj}} \psi_{Z_{i_{pj}}}^{\delta_{pj}}) Z_{i_{pj}}^{\epsilon_{pj}}
			h_{pj} t_{pj} = 1.
		\end{align}
		For \(l \in \{1, \ \ldots, \ p\}\), define
		\begin{align*}
			W_l & = (Y_{i_{lj}}^{\epsilon_{lj}}) \psi_{Z_{i_{lj}}}^{\delta_{lj}} \psi_{t_{(l - 1)j}}
			\psi_{Z_{i_{(l - 1)j}}}^{\epsilon_{(l - 1)j}} \cdots \psi_{t_{1j}}
			\psi_{Z_{i_{1j}}}^{\epsilon_{1j}}, \\
			f_l & = (h_{lj}) \psi_{Z_{i_{lj}}}^{\epsilon_{lj}} \psi_{t_{(l - 1)j}} \cdots \psi_{t_{1j}}
			\psi_{Z_{i_{1j}}}^{\epsilon_{1j}}.
		\end{align*}
		By pushing all \(Y_i\)s and \(h_i\)s to the left within \eqref{virt_eqn_2},
		we obtain
		\begin{align}
			\label{virt_H_to_left_eqn}
			W_1 f_1 \cdots W_p f_p
			Z_{i_{1j}}^{\epsilon_{1j}} t_{1j} \cdots Z_{i_{pj}}^{\epsilon_{pj}} t_{pj}
			= 1.
		\end{align}
		As \(H\) is a finite index subgroup of \(G\), \(Ht\) is a recognisable
		subset of \(G\), for all \(t \in T\). For each coset \(Ht\) of \(H\), and
		each variable \(X_i\) let \(R_{ti} = R_{X_i} \cap (Ht)\). Note that each set
		\(R_{ti}\) is rational, since each \(R_{ti}\) is the intersection of a
		rational set with a recognisable set.

		By Lemma \ref{Grunschlag_rational_lem}, we have that for each \(t \in T\),
		\(R_{ti} = S_{ti} t\), for some rational subset \(S_{ti}\) of \(H\). For
		every \((u_1, \ \ldots, \ u_n) \in T^n\) that forms a solution to the
		\(Z_i\)s within a solution to (\ref{virt_H_to_left_eqn}), we have
		\(u_{i_1}^{\epsilon_1} t_1 \cdots u_{i_p}^{\epsilon_p} t_p \in H\). Let \(A
		\subseteq T^n\) be the set of all such \(n\)-tuples. If we plug a fixed
		choice of some \((u_1, \ \ldots, \ u_n) \in T^n\) into
		\eqref{virt_H_to_left_eqn}, we obtain the following system of
		\(\Omega\)-twisted equations in \(H\):
		\begin{align*}
			\bar{W_1}
			f_1 \cdots
			\bar{W_p}
			f_p
			u_{i_{1j}}^{\epsilon_{1j}} t_{1j} \cdots u_{i_{pj}}^{\epsilon_{pj}} t_{pj} = 1,
		\end{align*}
		where
		\[
			\bar{W_l} = (Y_{i_{lj}}^{\epsilon_{lj}}) \psi_{u_{i_{lj}}}^{\delta_{lj}} \psi_{t_{(l - 1)j}}
			\psi_{u_{i_{(l - 1)j}}}^{\epsilon_{(l - 1)j}} \cdots \psi_{t_{1j}}
			\psi_{u_{i_{1j}}}^{\epsilon_{1j}},
		\]
		is \(W_l\), with each \(Z_i\) being replaced by \(u_i\).  We can now apply
		the rational constraint \(S_{ti}\) to the variable \(Y_i\), and we have a
		system of equations \(\mathcal E_{H, (u_1, \ldots, u_n)}\) with rational
		constraints in \(H\). Let \(B_{(u_1, \ \ldots, \ u_n)}\) be the solution set
		to \(\mathcal E_{H, (u_1, \ldots, u_n)}\). It follows that
		\[
			\mathcal{S} = \bigcup_{(u_1, \ \ldots, \ u_n) \in A} \{(f_1 u_1, \ \ldots,
			\ f_n u_n) \mid (f_1, \ \ldots, \ f_n) \in B_{(u_1, \ \ldots, \ u_n)}\}.
		\]
	\end{proof}

	\begin{rmk}
		\label{virt_norm_form_rmk}
		Let \(G\) be a finite index overgroup of a group \(H\). We will define a
		normal form for \(G\), induced by an existing normal form on \(H\). Let
		\begin{itemize}
			\item \(\Sigma_H\) be a finite generating set for \(H\);
			\item \(\eta_H\) be a normal form for \(H\), with respect to \(\Sigma_H\);
			\item \(T\) be a (finite) right transversal for \(H\) in \(G\).
		\end{itemize}
		We will use \(\Sigma = \Sigma_H \sqcup T\) as our generating set for \(G\).
		Each \(g \in G\) can be written uniquely in the form \(g = h_gt_g\) for some
		\(h_g \in H\) and \(t_g \in T\). Define \(\eta \colon G \to
		(\Sigma^\pm)^\ast\) by
		\[
			g \eta = (h_g \eta_H) t_g.
		\]
		Note that if \(\eta_H\) is regular, then \(\eta\) is regular, as the
		concatentation of \(\im \eta_H\) with a finite language.

		As the following lemma shows, this construction also preserves the property
		of being quasigeodesic.
	\end{rmk}

	\begin{lem}
		\label{virt_norm_form_quasigeo_lem}
		Let \(G\), \(H\), \(\Sigma\), \(\Sigma_H\), \(T\), \(\eta\) and \(\eta_H\)
		be defined as in Remark \ref{virt_norm_form_rmk}. Then \(\eta_H\) is
		quasigeodesic if and only if \(\eta\) is quasigeodesic.
	\end{lem}

	\begin{proof}
		\((\Rightarrow)\): Suppose \(\eta_H\) is quasigeodesic. Then there exists
		\(\lambda > 0\), such that \(|h \eta_H| \leq \lambda |h|_{(H, \Sigma_H)} +
		\lambda\) for all \(h \in H\).

		For each \(t \in T\) and \(a \in \Sigma^\pm\), \(ta = \nu_{t, a}\), for some
		\(\nu_{t, a} \in \im \eta\). For all \(t, \ t' \in T\), we have \(tt' =
		\rho_{t, t'}\), for some \(\rho_{t, t'} \in \im \eta\). For each \(t^{-1}
		\in T^{-1}\), we have that \(t^{-1} =_G x_{t^{-1}}\), where \(x_{t^{-1}} \in
		(\Sigma_H^\pm \cup T)^\ast\). 	Let
		\[
			\mu = \max_{t^{-1} \in T^{-1}} |x_{t^{-1}}| + \max_{t, t' \in
			T}|\rho_{t, t'}| + \mathop{\max_{a \in \Sigma_H}}_{t \in T} |\nu_{t, a}|
		\]

		Let \(w \in (\Sigma^\pm)^\ast\) be a geodesic. We will convert \(w \in
		(\Sigma^\pm)^\ast\) into a word \(u\), such that \(u =_G w\) and \(u \in \im
		\eta\), and we will show that \(|u| \leq \mu^2 \lambda |w| + \mu^2
		\lambda\).

		We first replace each occurence of \(t^{-1} \in T^{-1}\) with the word
		\(x_{t^{-1}}\) within \(w\). Since \(|x_{t^{-1}}| \leq \mu\) for all
		\(t^{-1} \in T\), doing this will result in a new word \(w_1 \in
		(\Sigma_H^\pm \cup T)^\ast\), such that \(w_1 =_G w\), and \(|w_1| \leq \mu
		|w|\).

		We now modify \(w_1\) into a word \(w_2\) such that \(w_1 =_G w_2\), and
		\(w_2\) contains no subword of the form \(ta\) or \(tt'\), where \(t, \ t'
		\in T\) and \(a \in \Sigma_H^\pm\). For each subword \(ta\) of \(w\), we can
		replace \(ta\) with \(\nu_{t, a}\), and for every occurence of \(tt'\), we
		can replace this with \(\rho_{t, t'}\). Each time we do this, we increase
		the length of the word by at most \(\mu\). Repeating this process until no
		subwords of the form \(ta\) remain, will yield \(w_2\).

		To ensure we don't  need to do too many of these replacements to satisfy
		linear bound of the length of \(w_2\) in terms of \(w_1\), we will always
		apply the leftmost substitution possible. As every replacement invloves a
		letter \(t \in T\) at the beginning of a two-letter word, and results in a
		word with exactly one two-letter in \(T\) at the end, one `sweep' along
		\(w_1\) will be sufficient to reach a word where no substitutions are
		possible. It follows that we can make at most \(|w_1|\) replacements, and
		since each substitution increases the length by at most \(\mu\), we have
		that \(|w_2| \leq \mu |w_1|\).

		We have that \(w_2 = v t\), for some \(v \in (\Sigma_H^\pm)^\ast\), and some
		\(t \in T\). To convert \(w_3\) into \(u\), it remains to replace \(v\) with
		an equivalent word \(q \in \im \eta_H\). As \(\eta_H\) is quasigeodesic with
		the constant \(\lambda\), \(|q| \leq \lambda |v| + \lambda\). If we take \(u
		= qt\), then \(u\) is equivalent in \(G\) to \(w\), and \(u \in \im \eta\).
		Note also that \(|u| \leq \lambda |w_2| + \lambda\). Therefore
		\[
			|u| = \lambda |w_2| + \lambda \leq \mu \lambda |w_1| + \mu \lambda
			\leq \mu^2 \lambda |w| + \mu^2 \lambda.
		\]
		It follows that \(\eta_H\) is quasigeodesic, with respect to the constant
		\(\lambda \mu^2\).

		\((\Leftarrow)\): Suppose \(\eta\) is quasigeodesic, with respect to a
		constant \(\lambda > 0\). Let \(w \in (\Sigma_H^\pm)^\ast\) be a geodesic,
		\(u \in \im \eta_H\) be such that \(u =_H w\), and \(v \in
		(\Sigma^\pm)^\ast\) be a geodesic in \(G\), such that \(v =_G w\). Note that
		\(u \in \im \eta\). As \(\eta\) is quasigeodesic, \(|u| \leq \lambda |v| +
		\lambda\). Moreover, since \(|w|\) and \(|v|\) are both geodesic words
		representing elements that lie in \(H\), but \(v\) is over the generating
		set \(\Sigma_G\) that contains the generating set \(\Sigma_H\) for \(w\),
		\(|v| \leq |w|\). Thus \(|u| \leq \lambda |w| + \lambda\), as required.
	\end{proof}

	We can use Lemma \ref{virt_norm_form_quasigeo_lem} to show that passing to
	the Schreier normal form also preserves the property of being quasigeodesic.

	\begin{lem}
		\label{Schreier_norm_form_quasigeo_lem}
		Let \(G\) be a group, generated by a finite set \(\Sigma\), \(H\) be a
		finite index subgroup of \(G\), and \(T\) be a right transversal of \(H\) in
		\(G\), containing \(1\). Let \(Z\) be the Schreier generating set for \(H\).
		Fix a normal form \(\eta\) for \((G, \ \Sigma)\). If \(\eta\) is
		quasigeodesic with respect to \(\Sigma\), then the Schreier normal form with
		respect to \(\eta\) is quasigeodesic with respect to the Schreier
		generators.
	\end{lem}

	\begin{proof}
		Let \(\zeta\) be the Schreier normal form for \(H\), with respect to
		\(\eta\). We will show that the normal form from Remark
		\ref{virt_norm_form_rmk}, inherited from \(\zeta\), is quasigeodesic. The
		result will then follow by the backward direction of Lemma
		\ref{virt_norm_form_quasigeo_lem}. Since \(\eta\) is quasigeodesic, there
		exists \(\lambda > 0\), such that \(|g \eta| \leq \lambda |g|_{(G, \Sigma)}
		+ \lambda\) for all \(g \in G\).

		Let \(\xi\) denote the normal form from Remark \ref{virt_norm_form_rmk},
		inherited from \(\zeta\), with respect to the transversal \(T\). Let \(w \in
		(\Sigma^\pm)^\ast\) be geodesic. We have that there exists \(v \in \im
		\eta\), such that \(v =_G w\), and \(|v| \leq \lambda |w| + \lambda\). We
		also have that there exists \(t_0 \in T\) such that \(v t_0\) represents an
		element of \(H\). We can then convert this into Schreier normal form to give
		a word \(u\). Note that \(|u| \leq |v t_0|\).

		We also have that there exists \(t_1 \in T\), such that \(ut_1 =_G w\). Note
		that \(u t_1 \in \im \xi\). Combining our inequalities that relate \(u\),
		\(v\) and \(w\), gives:
		\[
			|u t_1| \leq |v t_0 t_1| = |v| + 2 \leq \lambda |w| + 2 \lambda.
		\]
		So \(\xi\) is quasigeodesic, with respect to a constant \(2 \lambda\). The
		result now follows by Lemma \ref{virt_norm_form_quasigeo_lem}.
	\end{proof}

	The following result is not new; a slightly different version of it is used
	implicitly to show systems of equations with rational constraints in virtually
	free groups have EDT0L solution languages. We use it here to show the same is
	true for virtually abelian groups.

	\begin{proposition}
		\label{virt_eqn_prop}
		Let \(G\) be a group with a finite index normal subgroup \(H\), and let
		\(\Omega\) be the group of automorphisms of \(H\) induced by conjugation by
		elements of \(G\). Suppose that solutions to systems of \(\Omega\)-twisted
		equations in \(H\) with rational constraints are EDT0L in \(\ns(f)\), where
		\(f \colon \mathbb{Z}_{\geq 0} \to \mathbb{Z}_{\geq 0}\). Then solutions
		to systems of equations with rational constraints in \(G\) are EDT0L in
		\(\ns(f)\) with respect to the normal form from Remark
		\ref{virt_norm_form_rmk}.
	\end{proposition}

	\begin{proof}
		We can use the same proof that is used in Proposition 3.9 and Lemma 3.11 in
		\cite{VAEP}, using Lemma \ref{lem:virtual_sols} in place of the Lemma 3.8
		used in \cite{VAEP}.
	\end{proof}

	In order to prove our results about wreath products and direct products, we
	need some normal forms on groups made using these constructions.

	\begin{rmk}
		\label{dir_prod_norm_form_rmk}
		Let \(H_1, \ \ldots, \ H_k\) be groups, with finite generating sets
		\(\Sigma_{H_1}, \ \ldots, \ \Sigma_{H_k}\), and normal forms \(\eta_{H_1}, \
		\ldots, \ \eta_{H_k}\), respectively. Let \(G = \prod_{i = 1}^k H_i\). We
		will use \(\Sigma = \Sigma_{H_1} \sqcup \cdots \sqcup \Sigma_{H_k}\) as a
		generating set for \(G\). Define the \(\eta \colon G \to (\Sigma^\pm)^\ast\)
		by
		\[
			(h_1, \ \ldots, \ h_k) \eta = (h_1 \eta_{H_1}) \cdots (h_k \eta_{H_k}).
		\]
		Since concatenations of regular languages are regular, if every
		\(\eta_{H_i}\) is regular, then \(\eta\) is a regular normal form.

		In addition, the length of any element \(g \in G\) with respect to
		\(\Sigma\) is just the sum of the lengths of the projection of \(g\) to each
		\(H_i\), and from this it follows that if every \(\eta_{H_i}\) is
		(quasi)geodesic, then so is \(\eta\).
	\end{rmk}

	\begin{rmk}
		\label{wreath_prod_norm_form_rmk}
		Let \(H\) be a group, and \(K\) be a finite group. Let \(\Sigma_H\) be a
		generating set for \(H\), and \(\eta_H\) be a normal form with respect to
		\(\Sigma_H\). We define a generating set and normal form for \(H \wr K\),
		using \(\Sigma_H\) and \(\eta_H\). Note that \(H \wr K\) contains \(\prod_{i
		= 1}^n H_i\) as a finite index subgroup, where \(n \in \mathbb{Z}_{> 0}\),
		and \(H_i \cong H\) for all \(i\). We endow \(\prod_{i = 1}^n H_i\) with a
		generating set and normal form using Remark \ref{dir_prod_norm_form_rmk}.
		After this, we can use the generating set and normal form from Remark
		\ref{virt_norm_form_rmk} for \(H \wr K\), with respect the generating set
		and normal form of \(\prod_{i = 1}^n H_i\).

		Since the two constructions we have used to produce a normal form for \(H
		\wr K\) preserve the properties of regular and quasigeodesic, if
		\(\eta_H\) is regular or quasigeodesic, then so is the normal form on \(H
		\wr K\).
	\end{rmk}

	We show that the class of groups with EDT0L solutions to systems of equations
	is closed under direct products, and wreath products with finite groups. We
	start with the latter. We refer the reader to \cite{groups_langs_aut} for the
	definition of a wreath product.

	We first consider the properties of the normal forms we will be using.

	\begin{lem}
		\label{wreath_prod_norm_form_reg_quasigeo_lem}
		Let \(H\) and \(\eta_H\) be as in Remark \ref{wreath_prod_norm_form_rmk}. If
		\(\eta_H\) is regular or quasigeodesic, then the normal form on \(H \wr K\)
		from Remark \ref{wreath_prod_norm_form_rmk} will be regular or
		quasigeodesic, respectively.
	\end{lem}

	\begin{proof}
		Recall that the normal form in Remark \ref{wreath_prod_norm_form_rmk} is
		created by using the normal form for direct products (Remark
		\ref{dir_prod_norm_form_rmk}), followed by the normal form for finite
		extensions \ref{virt_norm_form_rmk}. Since both of these constructions
		preserve the properties regular and quasigeodesic, the result follows.
	\end{proof}

	We can now show that equations in wreath products have the desired properties.

	\begin{proposition}
		\label{wreath_product_EDT0L_prop}
		Let \(H\) be a group such that solutions to systems of equations with
		respect to a normal form \(\eta_H\) are EDT0L in \(\ns(f)\), where \(f
		\colon \mathbb{Z}_{\geq 0} \to \mathbb{Z}_{\geq 0}\). Let \(K\) be a finite
		group. Then
		\begin{enumerate}
			\item Solutions to systems of equations in \(H \wr K\) are
			EDT0L in \(\ns(f)\), with respect to the normal form from Remark
			\ref{wreath_prod_norm_form_rmk};
			\item If \(\eta_H\) is regular or quasigeodesic, then the normal form on
			\(H \wr K\) will be regular or quasigeodesic, respectively.
		\end{enumerate}
	\end{proposition}

	\begin{proof}
		First note that (2) follows from Lemma
		\ref{wreath_prod_norm_form_reg_quasigeo_lem}.

		Let  \(A\) be the finite set that \(K\) acts on, and define \(H \wr K\) with
		respect to this action. Let \(H_1, \ \ldots, \ H_{|A|}\) be the isomorphic
		copies of \(H\). Using Proposition \ref{virt_eqn_prop}, it suffices to show
		that that solutions to systems of \(\Omega\)-twisted equations in \(G :=
		\prod_{i = 1}^{|A|} H_i\) are EDT0L in \(\ns(f)\), with respect to the
		normal form from Remark \ref{dir_prod_norm_form_rmk}, where \(\Omega\) is
		the set of automorphisms defined by permuting the \(H_i\)s.

		Consider a system \(\mathcal{E}\) of \(\Omega\)-twisted equations in \(G\)
		in \(n\) variables. As every element of \(G\) can be written in the form
		\(h_1 \cdots h_{|A|}\), where \(h_i \in H_i\) for all \(i\), for each
		variable \(X\) in \(\mathcal{E}\), we can define new variables \(X_i\) over
		\(H_i\) for each \(i\), by \(X = X_1 \cdots X_{|A|}\). As the elements of
		\(H_i\) commute with the elements of \(H_j\) for each \(i \neq j\), we can
		view any (untwisted) equation in \(G\) as a system of \(|A|\) equations in
		\(H\), each with disjoint set of variables, by projecting the original
		equation onto \(H_i\). The fact that these sets are disjoint follows from
		the fact that the \(i\)th equation in the system will be the projection to
		\(H_i\), whose variables will be of the form \(X_i\), for some original
		variable \(X\).

		Let \(\Phi \in \Omega\), and let \(\sigma \in S_n\) be the permutation
		induced by the action of \(\Phi\). Then \(X \Phi = (X_1 \cdots X_{|A|}) \Phi
		= X_{1 \sigma} \cdots X_{(|A|)\sigma}\). It follows that any twisted
		equation in \(G\) can be viewed as a system of \(|A|\) equations in \(H\),
		again using projections to each \(H_i\). The variables of each of the
		equations will no longer be disjoint, however. It follows that a system of
		twisted equations in \(G\) projects to a system of equations in \(H\). Thus,
		there exists a system \(\mathcal{F}\) of equations in \(H\) with solution
		set \(S_{\mathcal{F}}\), such that \(\mathcal{F}\) has \(|A|n\) variables,
		and each variable is assigned an index in \(\{1, \ \ldots, \ |A|\}\), such
		that precisely \(n\) variables have each index, and such that the solution
		language of \(\mathcal{E}\) is equal to
		\[
			\{x_{11} \cdots x_{1|A|} \# \cdots \# x_{n1} \cdots x_{n|A|} \mid
			(x_{1i}, \ \ldots, \ x_{ni}) \in S_{\mathcal{F}} \text{ with each variable
			indexed by } i \text{ for all } i\}.
		\]
		From our assumptions, we have that the solution language to \(\mathcal{F}\)
		is EDT0L, and can be constructed in \(\ns(f)\). It follows that the language
		\[
			L_i = \{x_{1i} \# \cdots \# x_{n_i} \mid (x_{1i}, \ \ldots, \ x_{ni})
			\in S_{\mathcal{F}} \text{ with each variable indexed by } i\}
		\]
		is EDT0L for each choice of \(i\), and constructible in \(\ns(f)\), using
		Lemma \ref{EDT0L_diff_hash_lem}, and then taking the image under an
		appropriate free monoid endomorphism with Lemma
		\ref{EDT0L_closure_properties_lem}. Lemma \ref{intermesh_EDT0L_lem} then
		shows that the solution language to \(\mathcal{E}\) is EDT0L in \(\ns(f)\).
	\end{proof}

	We conclude this section with the proof that direct products also preserve the
	property of having EDT0L solution languages.

	\begin{proposition}
		\label{dir_prod_EDT0L_prop}
		Let \(f \colon \mathbb{Z}_{\geq 0} \to \mathbb{Z}_{\geq 0}\). Let \(G\) and
		\(H\) be finitely generated groups where solutions to systems of equations
		are EDT0L in \(\ns(f)\). Then
		\begin{enumerate}
			\item  The same holds in \(G \times H\), with
			respect to the normal form from Remark \ref{dir_prod_norm_form_rmk};
			\item If the normal forms on \(G\) and \(H\) are regular or quasigeodesic,
			then the normal form on \(G \times H\) will be regular or quasigeodesic,
			respectively, with respect to the union of the generating sets for \(G\)
			and \(H\).
		\end{enumerate}
	\end{proposition}

	\begin{proof}
		Part (2) follows from Remark \ref{dir_prod_norm_form_rmk}.

		Let \(\Sigma_G\) be a finite generating set for \(G\), and \(\Sigma_H\) be a
		finite generating set for \(H\). We will use \(\Sigma = \Sigma_G \sqcup
		\Sigma_H\) as our generating set for \(G \times H\). Consider an equation
		\(\omega = 1\) in \(G \times H\). Let \(\mathcal{X}\) be the set of
		variables  in \(\omega\). We have that every element of \(G \times H\) can
		be expressed in the form \(gh\) for some \(g \in G\) and \(h \in H\). We can
		reflect this in the variables as well, by defining new variables \(X_G\)
		over \(G\) and \(X_H\) over \(H\), for each \(X \in \mathcal{X}\), such that
		\(X = X_G X_H\). Let \(\mathcal{X}_G = \{X_G \mid X \in \mathcal{X}\}\), and
		\(\mathcal{X}_H = \{X_H \mid X \in \mathcal{X}\}\).

		As elements of \(G\) commute with elements of \(H\), we can rearrange
		\(\omega = 1\) into the form \(\nu \zeta = 1\), where \(\nu \in
		(\Sigma_G^\pm \cup X_G^\pm)^\ast\) and \(\zeta \in (\Sigma_H^\pm \cup
		X_H^\pm)^\ast\). Consider a potential solution \((g_1 h_1, \ \ldots, \ g_n
		h_n)\) to \(\omega = 1\), where each \(g_i \in G\) and each \(h_i \in H\).
		We have that this is a solution if and only if \((g_1, \ \ldots, \ g_n)\) is
		a solution to the equation \(\nu = 1\), and \((h_1, \ \ldots, \ h_n)\) is a
		solution to the equation \(\zeta = 1\). Note that these are equations in
		\(G\) and \(H\), respectively.

		Let \(\mathcal{E}\) be a system of equations in \(G \times H\). It follows
		that there exist systems of equations in \(G\) and \(H\) with solution
		sets \(\mathcal{S}_{G}\) and \(\mathcal{S}_H\), such that the solution set
 		to \(\mathcal{E}\) equals
		\[
			\{(g_1 h_1, \ \ldots, \ g_n h_n) \mid (g_1, \ \ldots, \ g_n) \in
			\mathcal{S}_G, \ (h_1, \ \ldots, \ h_n) \in \mathcal{S}_H\}.
		\]
		If \(\mathcal{L}_G\) and \(\mathcal{L}_H\) are EDT0L solution languages
		corresponding to these systems in \(G\) and \(H\), respectively, it follows
		that the solution language to \(\mathcal{E}\) equals
		\[
			\{\omega_0 \nu_0 \# \cdots \# \omega_n \nu_n \mid \omega_0 \# \cdots \#
			\omega_n \in \mathcal{L}_G, \ \nu_0 \# \cdots \# \nu_n \in
			\mathcal{L}_H\}.
		\]
		The result now follows by Lemma \ref{intermesh_EDT0L_lem}.
	\end{proof}

\section{Virtually abelian groups}
	\label{virt_abelian_sec}

	In this section we expand the work of Evetts and the author \cite{VAEP}, to
	show that systems of equations with rational constraints in virtually abelian
	groups have EDT0L solution languages. We start by looking at rational sets in
	free abelian groups and show that systems of twisted equations in free abelian
	groups have EDT0L solution languages. After this, we use Proposition
	\ref{virt_eqn_prop} to complete the proof. Grunschlag's result about rational
	sets in finite index subgroups (Lemma \ref{Grunschlag_rational_lem}) is used
	to allow this generalisation.

	We start with the definition of semilinear sets, used to give a description
	of rational subsets of free abelian groups.

	\begin{dfn}
		Let \(k \in \mathbb{Z}_{> 0}\). A subset of \(\mathbb{Z}^k\)
		that can be written in the form
		\[
			\{\mathbf{c}_1 n_1 + \cdots + \mathbf{c}_r n_r + \mathbf{d} \mid n_1, \
			\ldots, \ n_r \in \mathbb{Z}_{\geq 0}\},
		\]
		where \(\mathbf{c}_i, \ \mathbf{d} \in \mathbb{Z}^k\) for all \(i\), is
		called \textit{linear}. A finite union of linear sets is called
		\textit{semilinear}.
	\end{dfn}

	Showing that semilinear sets are rational is immediate from the definition.
	The converse is also true, thus giving a full classification of rational sets
	in free abelian groups.

	\begin{lem}[\cite{Eilenberg_schutzenberger}]
		A subset of a free abelian group is rational if and only if it is
		semilinear.
	\end{lem}

	Since semilinear sets are defined in terms of equations and inequalities, we
	can use this to describe sets of solutions to systems of twisted equations
	with rational constraints in free abelian groups.

	\begin{lem}
		\label{rational_sets_to_lin_eqs_lem}
		Let \(\mathcal{S}_\mathcal{E}\) be the solution set of a finite system
		\(\mathcal{E}\) of twisted equations in \(\mathbb{Z}^k\) in \(n\) variables
		with rational constraints. Then there is a finite disjunction
		\(\mathcal{F}\) of finite systems of equations, and inequalities of the form
		\(X \geq 0\), for some variable \(X\), in \(\mathbb{Z}\) with \(kn\)
		variables and solution set \(\mathcal{S}_\mathcal{F}\) such that
		\[
			\mathcal{S}_\mathcal{E} = \{((x_1, \ \ldots, \ x_k), \ \ldots, \ (x_{(k -
			1)n + 1}, \ \ldots, \ x_{kn})) \mid (x_1, \ \ldots, \ x_{kn}, \ y_1, \
			\ldots, \ y_r) \in \mathcal{S}_\mathcal{F}\}.
		\]
	\end{lem}

	\begin{proof}
		Converting the twisted system into a system over \(\mathbb{Z}\) can be done
		by replacing each variable \(\mathbf{X}\) over \(\mathbb{Z}^k\) with \(k\)
		variables \(X_1, \ \ldots, \ X_k\) over \(\mathbb{Z}\), and considering the
		system that results from looking at each coordinate individually. A full
		proof of this can be found in \cite{VAEP}. Now consider the
		membership problem of a variable \(\mathbf{X}\) into a linear set \(R =
		\{\mathbf{c}_1 n_1 + \cdots + \mathbf{c}_r n_r + \mathbf{d} \mid n_1, \
		\ldots, \ n_r \in \mathbb{Z}_{\geq 0}\}\) (we will then generalise to
		semilinear).

		Write \(\mathbf{c}_i = (c_{i1}, \ \ldots, \ c_{ik})\) and \(\mathbf{d} =
		(d_1, \ \ldots, \ d_k)\). Consider the following system of equations and
		inequalities over \(\mathbb{Z}\).
		\begin{align*}
			Y_i \geq 0, \quad X_j = c_{1j} Y_1 + \cdots c_{rj} Y_r
			+ d_j
		\end{align*}
		for all \(i \in \{1, \ \ldots, \ r\}\), and \(j \in \{1, \ \ldots, \ k\}\),
		where \(Y_1, \ \ldots, \ Y_r\) are new variables over \(\mathbb{Z}\). We
		have that \((x_1, \ \ldots, \ x_k) \in \mathbb{Z}^k\) occurs within a
		solution \((x_1, \ \ldots, \ x_k, \ y_1, \ \ldots, \ y_k)\) to the above
		system, if and only if \((x_1, \ \ldots, \ x_k) \in R\).

		The result follows from the fact that the solution set to a disjunction of
		systems is just the union of the solution sets, so if we take the
		disjunction of the systems obtained from each linear set used in the finite
		union of a semilinear set, we obtain the desired disjunction.
	\end{proof}

	We are now in a position to describe the solution language to a system of
	twisted equations with constraints in a free abelian group, using an EDT0L
	system.

	\begin{lem}
		\label{free_abelian_EDT0L_lem}
		Solutions to systems of twisted equations with rational constraints in a
		free abelian group are EDT0L in non-deterministic quadratic space, with
		respect to free abelian equation length, and the standard normal form.
	\end{lem}

	\begin{proof}
		We will use \(\Sigma = \{a_1, \ \ldots, \ a_k\}\) to denote the standard
		generating set for \(\mathbb{Z}^k\). Let \(\mathcal{E}\) be a system of
		equations in \(\mathbb{Z}^k\) with solution language \(L\). By Lemma
		\ref{rational_sets_to_lin_eqs_lem}, there is a disjunction \(\mathcal{F}\)
		of systems of equations, and inequalities of the form \(X \geq 0\), in
		\(\mathbb{Z}\), with set of solutions \(\mathcal{S}_\mathcal{F}\), such that
		\[
			L = \{a_1^{x_1} \cdots a_k^{x_k} \# \cdots \# a_1^{x_{(k - 1)n + 1}}
			\cdots a_k^{x_{kn}} \mid (x_1, \ \ldots, \ x_{kn}, \ y_1, \ \ldots, \ y_r)
			\in \mathcal{S}_\mathcal{F}\}.
		\]
		Consider the following language
		\[
			M = \{a_1^{x_1} \# \cdots \# a_k^{x_k} \# \cdots \# a_1^{x_{(k - 1)n + 1}}
			\cdots \# a_k^{x_{kn}} \# b_1^{y_1} \# \cdots \# b_r^{y_r} \mid (x_1, \
			\ldots, \ x_{kn}, \ y_1, \ \ldots, \ y_r) \in \mathcal{S}_\mathcal{F}\}.
		\]
		We will start by showing that \(M\) is EDT0L. First note that as finite
		unions of EDT0L languages are EDT0L, we can assume \(\mathcal{F}\) is a
		single system of equations and inequalities, rather than a disjunction of
		systems. Let \(m\) be the number of inequalities of the form \(X \geq 0\)
		within \(\mathcal{F}\).

		We will proceed by induction on \(m\). If \(m = 0\), then \(\mathcal{F}\) is
		a system of equations in \(\mathbb{Z}\), and thus the solutions are EDT0L in
		\(\ns(n^2)\), by \cite{VAEP}. Inductively suppose \(M\) is EDT0L, and an
		EDT0L system is constructible in \(\ns(n^2)\), when \(m = k\), where \(k \in
		\mathbb{Z}_{\geq 0}\). If \(m = k + 1\), then \(\mathcal{F}\) can be
		obtained from a system of equations and inequalities \(\mathcal{G}\), with
		the addition of a single inequality \(X \geq 0\). By our inductive
		hypothesis, the solution language of \(\mathcal{G}\) is EDT0L, and an EDT0L
		system can be constructed in \(\ns(n^2)\). The addition of the inequality
		\(X \geq 0\), can be achieved by intersecting the solution language of
		\(\mathcal{G}\) with the regular language
		\[
			(\Sigma^\pm)^\ast \# (\Sigma^\pm)^\ast \# \cdots \# \{a_1^\pm\}^\ast
			\cdots \{a_{j - 1}^\pm\}^\ast \{a_j\}^\ast \{a_{j + 1}^\pm\}^\ast \cdots
			\{a_k^\pm\}^\ast \# (\Sigma^\pm)^\ast \# \cdots \# (\Sigma^\pm)^\ast,
		\]
		where \(j\) is the free abelian generator in the correct position
		corresponding to \(X\). The fact that \(M\) is EDT0L, and an EDT0L system is
		constructible in \(\ns(n^2)\) now follows from Lemma
		\ref{EDT0L_closure_properties_lem}.

		For each \(i \in \{0, \ \ldots, \ n - 1\}\), let
		\[
			u_i = a_1^{x_{ik + 1}} \#_{ik + 1} \cdots \#_{ik + k - 1} a_k^{x_{ik + k}}
			\#_{ik + k}.
		\]

		Since \(M\) is EDT0L, it follows by Lemma \ref{EDT0L_diff_hash_lem}, that
		\[
			M' = \{u_1 \cdots u_{n - 1} b_1^{y_1} \#_{kn + 1} \cdots \#_{kn + r - 1}
			b_r^{y_r} \#_{kn + r} \mid (x_1, \ \ldots, \ x_{kn}, \ y_1, \ \ldots, \
			y_r) \in \mathcal{S}_\mathcal{F}\}
		\]
		is EDT0L, and a system is constructible in \(\ns(n^2)\). In order to show
		that \(L\) is EDT0L, it suffices apply a sequence of free monoid
		homomorphisms to \(M'\) to obtain \(L\). Firstly, apply the homomorphism
		defined by mapping each of \(b_1, \ \ldots, \ b_r\) and \(\#_{kn}, \ \ldots,
		\ \#_{kn + r}\) to \(\varepsilon\) to \(M'\) to obtain
		\[
			M'' = \{a_1^{x_1} \#_1 \cdots \#_{k - 1} a_k^{x_k} \#_k \cdots \#_{kn - 1}
			a_k^{x_{kn}} \mid (x_1, \ \ldots, \ x_{kn}, \ y_1, \ \ldots, \ y_r) \in
			\mathcal{S}_\mathcal{F}\}.
		\]
		Now apply the homomorphism which maps each \(\#_i\) to \(\varepsilon\), with
		the exception of \(\#_k, \ \ldots, \ \#_{k(n - 1)}\). These will
		instead be mapped to \(\#\). The image of \(M''\) under this homomorphism is
		\(L\). The result now follows from Lemma \ref{EDT0L_closure_properties_lem}.
	\end{proof}

	We now have everything needed to show the following.

	\begin{theorem_no_number}[B]
		Solutions to a system of equations with rational constraints in a virtually
		abelian group are EDT0L in non-deterministic quadratic space, with respect
		to virtually abelian equation length, and with respect to the regular
		quasigeodesic normal form from Remark \ref{virt_norm_form_rmk}, induced by
		the standard normal form on free abelian groups.
	\end{theorem_no_number}

	\begin{proof}
		This fact that the solutions are EDT0L in \(\ns(n^2)\) follows from Lemma
		\ref{free_abelian_EDT0L_lem} and Proposition \ref{virt_eqn_prop}. The fact
		that the normal form is regular and quasigeodesic follows from Remark
		\ref{virt_norm_form_rmk}, and Lemma \ref{virt_norm_form_quasigeo_lem},
		respectively, together with the fact that the standard normal form on a free
		abelian group is regular and quasigeodesic.
	\end{proof}

\section{Recognisable constraints and finite index subgroups}
	\label{recognisable_constraints_sec}

	This section is used to show Proposition \ref{fin_index_EDT0L_prop}, that is,
	that the class of groups where systems of equations have EDT0L solutions is
	closed under passing to finite index subgroups. We use recognisable
	constraints to show this fact, by first proving that the addition of
	recognisable constraints to a system of equations with an EDT0L solution set
	does not change the fact that the solution set is EDT0L with respect to the
	ambient normal form of the group. We can then use the fact that finite index
	subgroups are recognisable, however the resulting language will be expressed
	as words over the generators for the ambient group. Expressing solutions to
	the finite index subgroup as words over one of its own generating sets, such
	as the Schreier generators, requires additionl arguments.

	We start by showing that the addition of recognisable constraints to systems
	of equations in a group preserves the property that all such systems have
	EDT0L solution languages.

	\begin{proposition}
		\label{eqn_add_constraints_prop}
		Let \(G\) be a finitely generated group such that solutions to systems of
		equations are EDT0L in \(\ns(f)\) with respect to some normal form \(\eta\),
		where \(f \colon \mathbb{Z}_{\geq 0} \to \mathbb{Z}_{\geq 0}\). Then
		solutions to systems of equations in \(G\) with recognisable constraints are
		EDT0L in \(\ns(f)\), with respect to \(\eta\).
	\end{proposition}

	\begin{proof}
		Let \(\Sigma\) be a finite generating set for \(G\), and fix a normal form
		\(\eta\) for \((G, \ \Sigma)\) such that solution languages to systems of
		equations are EDT0L. Consider a system of equations \(\mathcal{E}\) with
		recognisable constraints in \(G\) with \(n\) variables. Let \(R_1, \ \ldots,
		\ R_n\) denote the constraints. Let \(L\) be the solution language to
		\(\mathcal{E}\) with the constraints removed. Let \(\pi \colon \Sigma^\ast
		\to G\) be the natural homomorphism. Note that \(S = (R_1 \pi^{-1}) \# (R_2
		\pi^{-1}) \# \cdots \# (R_n \pi^{-1})\) is a regular language. By Lemma
		\ref{EDT0L_closure_properties_lem}, \(L \cap S\) is EDT0L, and if an EDT0L
		system for \(L\) is constructible in \(\ns(f)\), then one for \(L \cap S\)
		is also constructible in \(\ns(f)\). As \(L \cap S\) is the solution
		language to \(\mathcal{E}\), the results follow.
	\end{proof}

	Since finite index subgroups are examples of recognisable sets, we can show
	the following.

	\begin{lem}
		\label{finite_index_subgroup_over_G_lem}
		Let \(f \colon \mathbb{Z}_{\geq 0} \to \mathbb{Z}_{\geq 0}\). Let \(G\)
		be a finitely generated group where solutions to systems of equations are
		EDT0L in \(\ns(f)\), with respect to some normal form \(\eta\), and let
		\(H\) be a finite index subgroup of \(G\). Let \(\mathcal E\) be a system of
		equations in \(G\). Then
		\begin{enumerate}
			\item the language of all solutions to \(\mathcal E\) that lie in \(H\)
			forms an EDT0L language, with respect to the normal form \(\eta\)
			restricted to \(H\);
			\item The EDT0L system for this language is constructible in \(\ns(f)\).
		\end{enumerate}

	\end{lem}

	\begin{proof}
		In order to restrict our solutions to \(H\), we add the constraint that
		every variable lies in \(H\), which is a recognisable subset of \(G\). The
		results now follow from Proposition \ref{eqn_add_constraints_prop}.
	\end{proof}

	%
	%
	%

	\begin{proposition}
		\label{fin_index_EDT0L_prop}
		Let \(f \colon \mathbb{Z}_{\geq 0} \to \mathbb{Z}_{\geq 0}\). Let \(G\)
		be a group where solutions to systems of equations are EDT0L in \(\ns(f)\),
		with respect to a normal form \(\eta\). Then the same holds in any finite
		index subgroup of \(G\) with respect to the Schreier normal form, inherited
		from \(\eta\).
	\end{proposition}

	\begin{proof}
		Let \(X\) be a finite generating set for \(G\), \(T\) be a right transversal
		for \(H\), and \(Z\) be the Schreier generating set for \(H\). Let \(\zeta\)
		be the Schreier normal form for \(H\).

		Fix a system \(\mathcal E\) of equations in \(H\). This can be considered
		as a system of equations in \(G\), with the restriction that the solutions
		must lie in \(H\). Let \(L\) be the solution language to \(\mathcal E\)
		when expressed as words over \(G\) using the normal form \(\eta\); that is
		\[
			L = \{(g_1 \eta) \# \cdots \# (g_n \eta) \mid (g_1, \ \ldots, \ g_n)
			\text{ is a solution to } \mathcal E\}.
		\]
		Note that we require that solutions lie in \(H\), as \(\mathcal E\) is a
		system over \(H\). By Lemma \ref{finite_index_subgroup_over_G_lem}, \(L\) is
		an EDT0L language over an alphabet \(\Sigma\). Let \(\mathcal H = (\Sigma, \
		C, \ \omega, \ \mathcal R)\) be an EDT0L system for \(L\) that is
		constructible in \(\ns(f)\). Let \(B \subseteq \End(C^\ast)\) be
		the (finite) set over which \(\mathcal R\) is a regular language.

		By Lemma \ref{n_hash_sep_lem}, we can assume our start word is of the form
		\(\omega_1 \# \cdots \# \omega_n\). By adding new letters \(\perp_1 \cdots
		\perp_n\) to \(C\), and preconcatenating the rational control by the
		endomorphism defined by \(\perp_i \mapsto \omega_i\) for all \(i\), we can
		assume our start word is of the form \(\perp_1 \# \cdots \# \perp_n\). Note
		that as we can easily construct our new start word from our existing one,
		this will not affect space complexity.

		We will construct a new EDT0L system from \(\mathcal H\). Our extended
		alphabet will be letters in \(C\) transversal element \(t \in T\). Define
		\[
			C_\text{ind} = \{c^{t, a} \mid c \in C, \ t \in T, \ a \in X^\pm\} \cup
			\{\perp_1, \ \ldots, \ \perp_n\}.
		\]
		Our alphabet will be \(\Sigma_\text{ind} \cup \{\#\} = \{a^{t, a} \mid a \in
		X^\pm, \ t \in T\} \cup \{\#\}\). Our start word will be \(\perp_1 \# \cdots
		\# \perp_n\). We define our rational control as follows. For each \(\phi \in
		B\), define \(\Phi_\phi\) to be the set of all \(\psi \in
		\End(C_\text{ind}^\ast)\) such that
		\[
			c^{t, a} \psi = x_1^{t_1, b_1} \cdots x_k^{t_k, b_k},
		\]
		where \(c \psi = x_1 \cdots x_k\), with every \(x_i \in C\), each \(t_i \in
		T\), \(t_1 = t\), and \(\overline{t_i b_i} = t_{i + 1}\). Let \(\mathcal
		R_1\) be the regular language obtained by replacing each occurence of \(\phi
		\in B\) with the finite set \(\Phi_\phi\). Let \(t_0\) be the unique element
		in \(T \cap H\). Let \(\Psi \subseteq \End(C_\text{ind}^\ast)\) be the set
		of all \(\psi\) defined by
		\[
			\perp_i \psi = \perp_i^{t_0, \ a_i},
		\]
		for some \(a_1, \ \ldots, \ a_n \in X^\pm\). Define \(\mathcal G =
		(\Sigma_\text{ind}, \ C_\text{ind}, \	\perp_1 \# \cdots \# \perp_n, \ \Psi
		\mathcal R_1)\). By construction, \(\mathcal G\) accepts words in \(L\),
		where each letter, excluding \(\#\), has an index \((t, a) \in T \times
		X^\pm\), and such that for each indexed word \(w = a_1^{t_1, a_1} \cdots
		a_k^{t_k, a_k}\), the following hold:
		\begin{enumerate}
			\item \(t_1 = t_0\);
			\item \(t_i a_i = t_{i + 1}\) for all \(i\).
		\end{enumerate}
		To show that the solution language to
		\(\mathcal E\) is EDT0L with respect to \(\zeta\), it remains to apply the
		free monoid homomorphism \(\theta \colon \Sigma_\text{ind}^\ast \to (Z^\pm
		\cup \{\#\})^\ast\) to \(L(\mathcal G)\), defined by
		\[
			a^{t, a} \mapsto at \overline{at}^{-1}.
		\]
		It now remains to show that this EDT0L system can be constructed in
		\(\ns(f)\). By Lemma \ref{EDT0L_closure_properties_lem}, applying the
		homomorphism \(\theta\) does not affect the space complexity, so it is
		sufficient to show that \(\mathcal G\) is constructible in \(\ns(f)\). The
		number of indices we use is \(2|X| |T|\), which is constant, as it is based
		only on the group \(H\). It follows that we can write down \(C_\text{ind}\)
		and \(\Sigma_\text{ind}\) in \(\ns(f)\). The set \(\Psi\) is again only
		based on \(|X|\), and so to show our rational control is constructible in
		\(\ns(f)\), it suffices to prove that \(\mathcal R_1\) is.

		Note that \(|\Phi_\phi|\) is again only based on \(|X||T|\), and so is
		constant. We construct \(\mathcal R_1\) by proceeding with the procedure we
		used to construct \(\mathcal R\), except whenever we would add an edge
		labelled \(\phi \in B\) between two states, we add edges labelled with all
		of \(\Phi_\phi\) between the same states. We can compute \(\Phi_\phi\) each
		time we need it, so we need only record the information we used to construct
		\(\mathcal R\). We can conclude that \(\mathcal G\) is constructible in
		\(\ns(f)\), and so the language of solutions to \(\mathcal E\) is EDT0L in
		\(\ns(f)\).
	\end{proof}

\section{Virtually direct products of hyperbolic groups}
	\label{virt_direct_prod_sec}

	In this section, we show that solution languages to systems of equations in
	groups that are virtually direct products of hyperbolic groups are EDT0L. We
	adapt the method that Ciobanu, Holt and Rees use to show that the
	satisfiability of systems of equations in these groups is decidable
	\cite{equations_VDP}. For an introduction to hyperbolic groups, we
	refer the reader to Chapter 6 of \cite{groups_langs_aut}.

	We start with some lemmas needed to prove this result. The following lemma
	gives an embedding as a finite index subgroup of a group that is virtually a
	direct product of hyperbolic groups, into a direct product of groups where
	equations are better understood.

	\begin{lem}[Lemma 3.5 in \cite{equations_VDP}]
		\label{virt_DP_to_WP_lem}
		Let \(G\) be a group that contains a group of the form \(K_1 \times \cdots
		\times K_n\) as a finite index normal subgroup, such that every conjugate of
		each of the subgroups \(K_i\) lies in the set \(\{K_1, \ \ldots, \ K_n\}\).
		Then
		\begin{enumerate}
			\item If the groups \(K_i\) are all conjugate to each other, then \(G\) is
			isomorphic to a finite index subgroup of \(J \wr P\), where \(J \cong
			N_G(K_1)/ (K_2 \times \cdots \times K_n)\) contains a finite index
			subgroup isomorphic to \(K_1\), and \(P\) is finite;
			\item Suppose \(K_1, \ \ldots, \ K_k\) are representatives of the
			conjugacy classes of \(K_1, \ \ldots, \ K_n\) within \(G\). Then \(G\)
			is isomorphic to a finite index subgroup of a direct product
			\(W_1 \times \cdots \times W_k\), where \(W_i = J_i \wr P_i\), \(J_i\)
			contains \(K_i\) as a finite index subgroup, and \(P_i\) is finite, for
			all \(i\).
		\end{enumerate}
	\end{lem}

	We define a normal form for groups that are virtually direct products.

	\begin{rmk}
		\label{virt_dir_prod_norm_form}
		Let \(G\) be a group that has a finite index subgroup of the form \(K_1
		\times \cdots \times K_n\). Fix a finite generating set \(\Sigma_{K_i}\),
		and normal form \(\eta_{K_i}\) for each \(K_i\). Using Lemma
		\ref{virt_DP_to_WP_lem}, \(G\) embeds as a finite index subgroup of \(W_1
		\times \cdots \times W_k\), where \(W_i = J_i \wr P_i\), \(K_i\) embeds as a
		finite index subgroup of \(J_i\), and \(P_i\) is finite.
		\begin{itemize}
			\item We start by defining a generating set and normal form for each
			\(J_i\). Since \(J_i\) contains \(K_i\) as a finite index subgroup, we can
			use the generating set and normal form from Remark
			\ref{virt_norm_form_rmk}, induced by \(\Sigma_{K_i}\) and \(\eta_{K_i}\).
			We will denote this generating set and normal form using \(\Sigma_{J_i}\)
			and \(\eta_{J_i}\), respectively;
			\item Using \(\Sigma_{J_i}\) and \(\eta_{J_i}\), we can use the generating
			set and normal form defined in Remark \ref{wreath_prod_norm_form_rmk} to
			define a normal form for each \(W_i = J_i \wr P_i\). Using these
			generating sets and normal forms, Remark \ref{dir_prod_norm_form_rmk}
			gives us a generating set \(\Delta\) and a normal form \(\mu\) for \(W_1
			\times \cdots \times W_k\);
			\item As \(G\) embeds as a finite index subgroup of \(W_1 \times \cdots
			\times W_k\), we can use the Schreier generating set \(Z\) and normal form
			\(\zeta\) on \(G\), induced by \(\Delta\) and \(\mu\).
		\end{itemize}
	\end{rmk}

	\begin{lem}
		\label{virt_dir_prod_norm_form_quasi_reg_lem}
		Let \(G\) be a group that has a finite index subgroup of the form \(K_1
		\times \cdots \times K_n\), and let \(\eta_{K_i}\) be defined as in Remark
		\ref{virt_dir_prod_norm_form}. Let \(\zeta\) be the normal form on \(G\) from
		Remark \ref{virt_dir_prod_norm_form}. If each \(\eta_{K_i}\) is regular or
		quasigeodesic, then \(\zeta\) is regular or quasigeodesic, respectively.
	\end{lem}

	\begin{proof}
		Since each of the constructions we have used to create \(\zeta\) preserve
		the properties of being regular and quasigeodesic (Lemma
		\ref{Schreier_norm_form_reg_lem}, Remark \ref{dir_prod_norm_form_rmk}, Lemma
		\ref{wreath_prod_norm_form_reg_quasigeo_lem}, Lemma
		\ref{virt_norm_form_quasigeo_lem}, Lemma
		\ref{Schreier_norm_form_quasigeo_lem}), if every \(\eta_{K_i}\) is regular
		or every \(\eta_{K_i}\) is quasigeodesic, then \(\zeta\) will be regular or
		quasigeodesic, respectively.
	\end{proof}

	We now use Lemma \ref{virt_DP_to_WP_lem} to show that the group that is
	virtually a direct product has an EDT0L solution language, subject to
	conditions on the groups it is virtually a direct product of.

	\begin{proposition}
		\label{VDP_prop}
		Let \(G\) be a group that contains a group of the form \(K_1 \times \cdots
		\times K_n\) as a finite index normal subgroup, such that every conjugate of
		each of the subgroups \(K_i\) lies in the set \(\{K_1, \ \ldots, \ K_n\}\).
		Let \(f \colon \mathbb{Z}_{\geq 0} \to \mathbb{Z}_{\geq 0}\).
		\begin{enumerate}
			\item If solutions to systems of equations are EDT0L in \(\ns(f)\) in each
			group in \(\FIN(K_i)\), with respect to a normal form \(\eta_{K_i}\), then
			solutions systems of equations in \(G\) are EDT0L \(\ns(f)\), with respect
			to the normal form \(\zeta\) from Remark \ref{virt_dir_prod_norm_form};
			\item If all the normal forms used in the groups in \(\FIN(K_i)\) are
			regular or quasigeodesic, then \(\zeta\) will be regular or quasigeodesic,
			respectively.
		\end{enumerate}

	\end{proposition}

	\begin{proof}
		By Lemma \ref{virt_DP_to_WP_lem}, we have that \(G\) embeds as a finite
		index subgroup into \(W_1 \times \cdots \times W_k\), where \(W_i = J_i \wr
		P_i\) for finite index overgroups \(J_i\) of \(K_i\), and finite groups
		\(P_i\). By Lemma \ref{fin_index_EDT0L_prop}, it suffices to show that
		solutions to systems of equations are EDT0L in \(\ns(f)\) in \(W_1 \times
		\cdots \times W_k\). The fact that solutions to systems of equations are
		EDT0L in \(\ns(f)\) in each of the groups \(W_i\) follows by our
		assumptions, together with Proposition \ref{wreath_product_EDT0L_prop}. We
		can then use Proposition \ref{dir_prod_EDT0L_prop} to show that the same
		holds in \(W_1 \times \cdots \times W_k\).

		Part (2) follows from Lemma \ref{virt_dir_prod_norm_form_quasi_reg_lem}.
	\end{proof}

	We now apply Proposition \ref{VDP_prop} to the specific case when the groups
	in the direct product comprise one virtually abelian group, and other
	non-elementary hyperbolic groups.

	\begin{lem}[Proposition 4.4 in \cite{equations_VDP}]
		\label{virt_DP_to_normal_subp_lem}
		Let \(A\) be a virtually abelian group, and let \(H_1, \ \ldots, \ H_n\) be
		non-elementary hyperbolic groups. Let \(G\) be a group with a finite index
		subgroup \(H\) that is isomorphic to \(A \times H_1 \times \cdots \times
		H_n\). Then \(G\) has a finite index normal subgroup isomorphic to \(B
		\times K_1 \times \cdots \times K_n\), where \(B\) is a finite index
		subgroup of \(A\), and each \(K_i\) is a finite index subgroup of \(H_i\),
		such that every conjugate of each of the subgroups \(K_i\) lies in the set
		\(\{K_1, \ \ldots, \ K_n\}\).
	\end{lem}

	We finally need the fact that languages of solutions to systems of equations
	in hyperbolic groups are EDT0L.

	\begin{lem}[\cite{eqns_hyp_grps}]
		\label{hyperbolic_EDT0L_lem}
		Solutions to a system of equations in any hyperbolic group are EDT0L in
		\(\ns(n^4 \log n)\), with respect to any finite generating set, and any
		quasigeodesic normal form. If the hyperbolic group is torsion-free, the
		solutions are EDT0L in \(\ns(n^2 \log n)\).
	\end{lem}

	We are now in a position to show that groups that are virtually direct
	products of hyperbolic groups have EDT0L languages of solutions. Since every
	hyperbolic group admits a regular geodesic normal form, if these normal forms
	are used to induce the normal forms in the hyperbolic groups, then the normal
	form in the virtually direct product will be quasigeodesic and regular.

	\begin{rmk}
		In the following theorem, our groups are constructed from virtually abelian
		groups and other groups. As such, we are measuring our input size using
		equation length, not virtually abelian equation length. However, we will
		continue to use space complexity results from \cite{VAEP} and Section
		\ref{virt_abelian_sec} that use virtually abelian length. This is okay,
		since virtually abelian equation length is approximately the log of equation
		length, and so the actual space complexity will be at least as small. It is
		possible that the space complexity will be a smaller than stated, but it
		will still be polynomial.
	\end{rmk}

	\begin{theorem}
		\label{VDP_hyp_thm}
		Let \(G\) be a group that is virtually \(A \times H_1 \times \cdots \times
		H_n\), where \(A\) is virtually abelian, and \(H_1, \ \ldots, \ H_n\) are
		non-elementary hyperbolic. Then
		\begin{enumerate}
			\item Solutions to systems of equations in \(G\) are EDT0L in \(\ns(n^4
			\log n)\), with respect to the normal form \(\zeta\) from Remark
			\ref{virt_dir_prod_norm_form};
			\item If, in addition, all of the groups \(H_i\) are torsion-free, then
			the solutions are EDT0L in \(\ns(n^2 \log n)\);
			\item The normal form \(\zeta\) can be chosen to be quasigeodesic and
			regular.
		\end{enumerate}

	\end{theorem}

	\begin{proof}
		We have from Lemma \ref{virt_DP_to_normal_subp_lem}, that \(G\) has a finite
		index subgroup isomorphic to \(B \times K_1 \times \cdots \times K_n\),
		where \(B\) is a finite index subgroup of \(A\), and each \(K_i\) is a
		finite index subgroup of \(H_i\), such that every conjugate of each of the
		subgroups \(K_i\) lies in the set \(\{K_1, \ \ldots, \ K_n\}\). We have that
		\(B\) is virtually abelian and the groups \(K_i\) are non-elementary
		hyperbolic. Thus, all groups in \(\FIN(B)\) are virtually abelian, and all
		groups in \(\FIN(K_i)\) are hyperbolic for each \(i\). We can equip each of
		these with a regular quasigeodesic normal form. Theorem B and Lemma
		\ref{hyperbolic_EDT0L_lem} imply that solutions to systems of equations are
		EDT0L in \(\ns(n^4 \log n)\) in all of these groups. The result then follows
		from Proposition \ref{VDP_prop}.
	\end{proof}

	We can reformulate Theorem \ref{VDP_hyp_thm} in the following way.

	\begin{cor}
		\label{main_thm_cor}
		Let \(G\) be a group that is virtually a direct product of hyperbolic groups
		(resp. torsion-free hyperbolic groups). Then the solutions to systems of
		equations in \(G\) are EDT0L in \(\ns(n^4 \log n)\) (resp. \(\ns(n^2 \log
		n)\)), with respect to the normal form from Remark
		\ref{virt_dir_prod_norm_form}, which can be constructed to be quasigeodesic
		and regular.
	\end{cor}

	As dihedral Artin groups are virtually a direct product of free groups, we
	have the following result. Note that the generating set and normal form will
	not be the standard Artin group ones; they are derived by taking the Schreier
	generators with respect to some finite index overgroup. As with Theorem
	\ref{VDP_hyp_thm}, we can choose the regular geodesic normal forms for the
	free groups that dihedral Artin groups are virtually a direct product of, to
	give a regular quasigeodesic normal form for these dihedral Artin groups.

	\begin{cor}
		\label{Artin_groups_cor}
		The solutions to systems of equations in dihedral Artin groups are EDT0L in
		\(\ns(n^2 \log n)\), with respect to the normal form from Remark
		\ref{virt_dir_prod_norm_form}, which can be constructed to be quasigeodesic
		and regular.
	\end{cor}

	\begin{proof}
		This follows from Corollary \ref{main_thm_cor}, together with the fact that
		dihedral Artin groups are virtually direct products of free groups
		(Lemma \ref{dihedral_artin_virt_dir_prod_lem}).
	\end{proof}

	\begin{rmk}
		The generating set and normal form from Remark \ref{virt_dir_prod_norm_form}
		will be the Schreier  generating set and normal form inherited from some
		finite index overgroup. This will not (necessarily) be a `sensible'
		generating set and normal form for groups that are virtually a direct
		product of hyperbolic groups, or any of the standard normal forms used in
		dihedral Artin groups.

		It is easy to change the generating set whilst preserving the property of
		EDT0L solutions. To add a (redundant) generator \(a\), one can use the
		existing normal form, which never uses \(a\), and so the solution language
		will be unchanged. To remove a redundant generator \(b\), one can apply the
		free monoid homomorphism that maps \(b\) to some word \(w_b\) over the
		remaining generators and inverses, that represents \(b\), to the solution
		language to remove all occurences of \(b\). Applying the free monoid
		homomorphism that maps \(b^{-1}\) to \(w_b^{-1}\) after this, will give a
		new solution language, with \(b\) removed from the generating set. As images
		of EDT0L languages under free monoid homomorphisms are EDT0L, this new
		solution language will also be EDT0L.

		Changing the normal form is more difficult. Section 5 of
		\cite{eqns_hyp_grps} contains a successful attempt at this for hyperbolic
		groups, which uses Ehrenfeucht and Rozenberg's Copying Lemma
		\cite{copying_lemma}; a common tool used to show preimages of EDT0L
		languages under free monoid homomorphisms are EDT0L in certain cases, along
		with a result about languages of quasigeodesics in hyperbolic groups.
		Languages of quasigeodesics in virtually abelian groups are not so well
		behaved, and any attempt to show that alternative normal forms work in many
		of the groups considered here will need an alternative approach.
	\end{rmk}

	\section*{Acknowledgements}
	 I would like to thank Laura Ciobanu for extremely helpful mathematical
	 discussions and detailed writing guidance. I would like to thank Alex Evetts
	 for help with proof reading, and his work on Lemma \ref{EDT0L_diff_hash_lem}.
	 I would like to thank Alex Bishop for suggestions that helped to improve the
	 proof of Proposition \ref{fin_index_EDT0L_prop}. I would like to thank Murray
	 Elder for help finding references. Finally I would like to thank the
	 anonymous reviewer for very detailed and helpful comments.

\bibliography{references}
\bibliographystyle{abbrv}
\end{document}